\documentclass{article}

\usepackage{graphicx}
\usepackage{subcaption}
\usepackage[english]{babel}
\usepackage{amsfonts}
\usepackage{amsmath}
\usepackage{amsthm}
\usepackage{authblk}
\newtheorem{theorem}{Theorem}
\newtheorem{lemma}{Lemma}
\newtheorem{defn}{Definition}
\newtheorem{property}{Property}
\newtheorem{remark}{Remark}
\newtheorem{proposition}{Proposition}

\begin{document}

\title{Digraph Arborescences and Matrix Determinants}

\author[1,2]{Sayani Ghosh}
\author[3]{Bradley S. Meyer}

\affil[1]{Schmid College of Science and Technology, Chapman University,
Orange, CA, USA} 
\affil[2]{Institute of Quantum Studies, Chapman University,
Orange, CA, USA}
\affil[3]{Department of Physics and Astronomy, Clemson University,
Clemson, SC, USA}

\maketitle

\begin{abstract}
We present a version of the matrix-tree theorem, which relates the determinant of a matrix to sums of weights of arborescences of its directed graph representation.  Our treatment allows for non-zero column sums in the parent matrix by adding a root vertex to the usually considered matrix directed graph.  We use our result to prove a version of the matrix-forest, or all-minors, theorem, which relates minors of the matrix to forests of arborescences of the matrix digraph.  We apply the theorems to calculations of the time-evolution of a system with discrete states and then consider two strategies using these theorems to compute determinants.
\end{abstract}


\section{Introduction}

The matrix-tree theorem, attributed to Tutte \cite{tutte_1948}, relates the number of spanning directed trees in a directed graph to the determinant of a minor of a zero-column-sum matrix.
Chen \cite{chen2012applied} and Chaiken \cite{chaikenSeth} generalized the theorem to include more minors of the determinant, which resulted in sums over directed forests of the parent directed graph.  These works and others (including \cite{MOON1994163}, \cite{chebotarev2006matrixforest},  and \cite{tutte2001graph}) also provided versions of the theorems for weighted directed graphs.  Recently, De Leenheer has provided an elegant proof of the matrix-tree theorem using the Cauchy-Binet formula \cite{doi:10.1137/19M1265193}.

While the matrix-tree and matrix-forest theorems are well studied, our goal is to present these theorems in a form we find amenable to calculations on general matrices.  Of particular note for the present work, Moon showed that determinants of a general matrix can be computed as sums over weights of functional digraphs, which are like directed trees but allow for loops \cite{MOON1994163}.  The loops account for the non-zero sum of a given column in the matrix.  We build on this idea, but instead of considering loops, we add a root vertex to the directed graph that accounts for the non-zero column sum.  With this modification and the Cauchy-Binet formula, we prove our version of the matrix-tree theorem.  A version of our proof that did not account for a general non-zero column sum was presented by Wang \cite{wang2009branchings}.

We then consider the case of {\em reduced} matrices that have one or more columns replaced by all zeros except for a given row that has a one.  This is related to the matrix-forest theorems \cite{chen2012applied, chaikenSeth, MOON1994163, chebotarev2006matrixforest}.  We provide straightforward proofs and explicit examples for both theorems, and we discuss how they can be used to compute matrix inverses.

As an illustration of how the theorems can be useful for calculations, we apply them to modeling of the time-evolution of a system with discrete states.  The theorems provide a means of calculating and interpreting the flow of probability from one state to another over a finite timestep.  They also provide a means of understanding the state population probabilities when the system achieves equilibrium in terms of spanning branchings in the digraph corresponding to the rate matrix describing the system

We finally turn to the general question of using the matrix-tree theorem for calculations of matrix determinants.  The number of arborescences needed to compute the determinant of an $n\times n$ matrix typically grows rapidly with $n$.  Direct calculation of the determinant of a matrix by the matrix-tree theorem can thus be challenging.  We discuss how our version of the matrix-tree theorem can be used to compute an approximation of the determinant by computation of the largest weight arborescences and how it can be useful for developing recursive strategies for computing determinants of diagonal matrices.

\section{Digraph Representation of a Matrix}
\label{sec:digraphs_and_matrices}

We begin by considering an $n \times n$ matrix $A = \left[a_{ij}\right]$.
The matrix elements $a_{ij}$ are taken to be
\begin{equation}
a_{ij} = 
\begin{cases}
-v_{ij}, & i\ne j, 1 \leq i, j \leq n\\
 \sum_{k = 1}^n v_{kj}, & i = j, 1 \leq i \leq n
\end{cases}
\label{eq:aij}
\end{equation}
The sum of the elements in column $j$ of the matrix $A$ is thus $v_{jj}$.
Any matrix may be written in the form given by Eq. (\ref{eq:aij}).  The
numbers $v_{ij}$ may themselves be sums.  We thus note that, in general, we
may have
\begin{equation}
v_{ij} = \sum_{\ell=1}^{N_{ij}}  u_{ij}^{(\ell)}
\label{eq:uij}
\end{equation}

We seek a representation of $A$ as a directed graph.
A graph $G = (V,E)$ is a set $V$ of vertices and a
set of edges $E$, which are two-element subsets of $V$.  An edge is thus
a line (segment) connecting two vertices.  A directed graph (digraph)
$\Gamma = (V,\cal{A})$ is a set $V$ of vertices and a set $\cal{A}$ of
arcs, which are ordered pairs of vertices.  In particular, an arc $(i,j)$
is an arrow directed from vertex $i$ to vertex $j$, where $i$ and $j$ are
both elements of $V$.

\begin{defn}[Matrix Digraph]
Given the $n \times n$ matrix $A$ defined in Eqs. (\ref{eq:aij}) and
(\ref{eq:uij}),
we draw a graph with $n+1$ vertices with labels ranging from $0$ to $n$.
For each term $-u_{ij}^{(\ell)}$ in matrix element $a_{ij}$ with $i \ne j$,
we draw an arc
from vertex $i$ to vertex $j$ and give the arc weight $u_{ij}^{(\ell)}$.
For each term $u_{ii}^{(\ell)}$ in matrix element $a_{ii}$ we draw an arc
from vertex $0$ to vertex $i$ with weight $u_{ii}^{(\ell)}$.  The resulting
graph is the {\bf Matrix Digraph}.
\label{defn:matdigraph}
\end{defn}

The vertex $0$ in the matrix digraph has no in arcs and is the {\it root} vertex
of the digraph.
Two properties of the matrix digraph are worth noting.
\begin{property}
If $N_{ij} > 1$, the matrix digraph has parallel arcs from vertex $i$ to $j$.
The graph is a multidigraph.
\end{property}
\begin{property}
Because arcs arising from $v_{ii}$ terms in the matrix have the root 
as their source and vertex $i$ as their target, and because all other arcs
have vertex $i$ as their source and vertex $j \neq i$ as their target,
the matrix digraph has no loops (arcs with the source and
target being the same vertex).
\label{prop:loop}
\end{property}

The total number of arcs in the matrix digraph is denoted
$m$ and a particular arc $k$ is denoted $e_k$.  The out,
or source, vertex of $e_k$ is denoted $s(e_k)$, and the in, or target,
vertex of $e_k$ is $t(e_k)$.  The weight of $e_k$ is denoted $w(e_k)$.
With these definitions, we note from Eqs. (\ref{eq:aij}) that
\begin{equation}
a_{ii} = \sum_k \delta_{t(e_k),i} w(e_k),
\label{eq:Aii2}
\end{equation}
where the sum runs over all arcs but the Kr\"onecker delta picks out
only those with vertex $i$ as the target.
Similarly, from Eq. (\ref{eq:aij}), we find
\begin{equation}
a_{ij} = -\sum_k \delta_{s(e_k),i} \delta_{t(e_k),j}
w(e_k),
\label{eq:Aij2}
\end{equation}
where, in this case, the sum runs over all arcs with vertex $i$
as the source and vertex $j$ as the target.

We now extend our matrix $A$
to include a row and column with index $0$.
We denote the extended matrix as $A'$.  The matrix elements of $A'$ are
still given by Eqs. (\ref{eq:Aii2}) and (\ref{eq:Aij2}), but $i$ and $j$ may
take on the value $0$.  We may see that $a_{i0} = 0$, since there are no
in arcs to the root vertex $0$.  We may also see that $a_{0i} = -v_{ii}$.

\begin{remark}
While the extended matrix $A'$ now has $n + 1$ rows and columns,
we can return
to the original matrix $A$ by striking the first row and column,
that is, by striking row $0$ and column $0$.of a zero-column-sum matrix. 
\label{rem:extended}
\end{remark}

We may now write the matrix $A'$ as the product of
an incidence matrix $M$ and a weight matrix $W$.

\begin{defn}[Incidence Matrix]
An incidence matrix $M$ for a digraph
is a matrix with number of rows equal to the number of vertices in
the digraph and number of columns equal to the number of arcs.
The elements of $M$ are
\begin{equation}
M_{i,k} = \delta_{t(e_k),i} - \delta_{s(e_k),i},
\label{eq:incidence_matrix}
\end{equation}
where $\delta_{i,j}$ is the usual Kr\"onecker delta.
\label{defn:M}
\end{defn}

For a column $k$ in incidence matrix $M$,
there is a $-1$ in the row corresponding to the source vertex of the
arc $e_k$ and a $1$ in the row corresponding to the target vertex of
arc $e_k$.  Since there are no loops (Property \ref{prop:loop}),
there are no columns with a $-1$ and $1$ in the same row.

\begin{defn}[Weight Matrix]
The weight matrix $W$ has a number of rows equal to the number of arcs
in the graph and a number of columns equal to the number of vertices.
The elements of $W$ are
\begin{equation}
W_{k,j} = \delta_{t(e_k),j}w(e_k)
\label{eq:weight_matrix}
\end{equation}
\label{defn:W}
\end{defn}

The $k$-th row in $W$ corresponds to the $k$-th arc in the graph.
Each row in $W$ has a single non-zero element located in the column
corresponding to the index of the invertex (that is, the target
vertex) of arc $k$.

\begin{lemma}
The extended matrix $A' = MW$, where $M$ is the incidence
matrix (Definition \ref{defn:M}) and $W$ is the weight matrix
(Definition \ref{defn:W}).
\label{lem:mw}
\end{lemma}

\begin{proof}
The $(i,j)$ element of the $(n + 1) \times (n + 1)$ matrix $MW$ is
\[
\left(MW\right)_{i,j} = \sum_k M_{i,k} W_{k,j}
\]
\begin{equation}
= \sum_k \delta_{t(e_k),i}\delta_{t(e_k),j}w(e_k)
- \sum_k \delta_{s(e_k),i}\delta_{t(e_k),j}w(e_k).
\label{eq:mw_def}
\end{equation}
If $i = j$, the second term in the sum in
Eq. (\ref{eq:mw_def}) is zero since the matrix digraph
contains no loops (Property \ref{prop:loop}) and,
hence, the source and target of any arc must be distinct.
In this case,
\begin{equation}
\left(MW\right)_{i,i} = \sum_k \delta_{t(e_k),i}w(e_k).
\label{eq:mw_ii}
\end{equation}
If $i \neq j$, the first term in the sum in
Eq. (\ref{eq:mw_def}) is zero since an arc cannot have two distinct targets.
In this case,
\begin{equation}
\left(MW\right)_{i,j\neq i} = 
- \sum_k \delta_{s(e_k),i}\delta_{t(e_k),j}w(e_k).
\label{eq:mw_ij}
\end{equation}
Comparison of Eqs. (\ref{eq:Aii2}) and (\ref{eq:mw_ii}) and
Eqs. (\ref{eq:Aij2}) and (\ref{eq:mw_ij}) show that
\begin{equation}
A' = M W.
\label{eq:DAMW}
\end{equation}
\end{proof}

Lemma \ref{lem:mw} is true for the case of the $(n + 1) \times (n + 1)$
extended matrix $A'$ that includes row and column $0$.
It is also true for the case of the $n \times n$ matrix that does not
include a row $0$ and column $0$ if we imagine striking row $0$ of
$M$ and column $0$ of $W$ (remark \ref{rem:extended}).

As an example to illustrate the ideas in this section, consider the $3 \times 3$ matrix
\begin{equation}
    A = \begin{pmatrix}
4 & -1 & -1\\
-1 & 4 & -3 \\
-1  & -2 & 5
\end{pmatrix}
\label{eq:A}
\end{equation}
with
\begin{equation}
    det(A) = 42
\label{eq:example_det}
\end{equation}

The matrix digraph for $A$ is shown in Fig. \ref{fig:3digraph}.  The arcs and their weights in the matrix digraph in 
Fig. \ref{fig:3digraph}, sorted by their invertex, are shown in Table \ref{tab:arcs}.

\begin{center}
\begin{figure}[t]
\centering
    \includegraphics[height=2in]{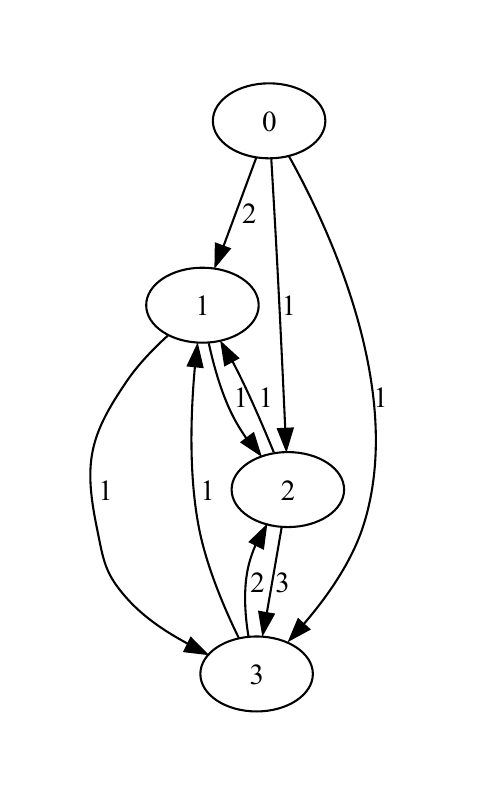}
    \caption{The matrix digraph corresponding to the matrix in Eq. (\ref{eq:A}).} \label{fig:3digraph}
\end{figure}
\end{center} 

\begin{table}
\begin{center}
    \begin{tabular}{ |c | c | c | }
        \hline
            Number & Arc & Weight\\
            \hline
            1 & (0, 1) & 2 \\
            2 & (2, 1) & 1 \\
            3 & (3, 1) & 1 \\
            4 & (0, 2) & 1 \\
            5 & (1, 2) & 1 \\
            6 & (3, 2) & 2 \\
            7 & (0, 3) & 1 \\
            8 & (1, 3) & 1 \\
            9 & (2, 3) & 3 \\
        \hline
    \end{tabular}
    \caption{\label{tab:arcs}Arcs in the matrix digraph in Fig. \ref{fig:3digraph}}        
\end{center}
\end{table}
The incidence matrix for the matrix digraph in Fig. \ref{fig:3digraph} is
\begin{equation}
    M = \begin{pmatrix}
-1 & 0 & 0 & -1 & 0 & 0 & -1 & 0 & 0\\
1 & 1 & 1 & 0 & -1 & 0 & 0 & -1 & 0\\
0 & -1 & 0 & 1 & 1 & 1 & 0 & 0 & -1\\
0 & 0 & -1 & 0 & 0 & -1 & 1 & 1 & 1\\
\end{pmatrix}
\label{eq:M}
\end{equation}
The weight matrix is
\begin{equation}
    W = \begin{pmatrix}
0 & 2 & 0 & 0\\
0 & 1 & 0 & 0\\
0 & 1 & 0 & 0\\
0 & 0 & 1 & 0\\
0 & 0 & 1 & 0\\
0 & 0 & 2 & 0\\
0 & 0 & 0 & 1\\
0 & 0 & 0 & 1\\
0 & 0 & 0 & 3\\
\end{pmatrix}
\label{eq:W}
\end{equation}

The extended matrix $A' = M W$.  It is easy to verify that the matrix $A$ is obtained by striking the first row and column of $A'$.

\section{The Matrix-Tree Theorem}

For a directed graph $\Gamma$, the indegree of any vertex is the number of arcs
entering that vertex while the outdegree is the number of
arcs exiting that vertex.  A branching $B$ on a graph is an acyclic subgraph
of $\Gamma$ that has no vertex with indegree larger than one.
If a digraph has $n$ vertices, a spanning branching, or arborescence,
has $n-1$ arcs.  The underlying graph is a tree (an acyclic connected graph).
The root of the arborescence is the one and only vertex with indegree zero.
A general branching has $n-1$ arcs or fewer and one or more roots.
Its underlying graph is a forest.

We compute the determinant of $A$,
denoted $det(A)$.
From Eq. (\ref{eq:DAMW}), we may write
\begin{equation}
det(A') = det(M W).
\label{eq:detDAMW}
\end{equation}
To calculate $det(M W)$,
we use the Cauchy-Binet formula, which states that,
if $M$ is an $n \times m$ matrix and $W$ is an $m \times n$ matrix,
the determinant of the $n \times n$ matrix $MW$ is
\begin{equation}
det(M W) = \sum_S det(M_S) det(W_S),
\label{eq:cauchy-binet}
\end{equation}
where the sum runs over all subsets $S$ of $\{1, ..., m\}$ with $n$ elements.
There are $C(m,n)$ such subsets $S$, where $C(m,n)$ is the usual binomial
coefficient.  $M_S$ is an $n \times n$
submatrix of $M$ consisting of the set of columns
$\{k\}$ of $M$ such that $k \in S$ while $W_S$ is an $n \times n$
submatrix of $W$ 
consisting of the set of rows $\{k\}$ of $W$ such that $k \in S$.

We apply Eq. (\ref{eq:cauchy-binet}) to $A$ for the
matrix digraph.
We consider the elements of $\{1, ..., m\}$ to be the labels of the
arcs in our digraph.  An $n$-element subset $S$ of
$\{1, ..., m\}$ thus corresponds to a subgraph of the digraph that
consists of a set of $n$ arcs $\{e_k\}$ such that $k \in S$.
We now consider striking row $0$ and column $0$, which will give the
determinant of the desired $n \times n$
matrix (remark \ref{rem:extended}).
This corresponds to striking row $0$ of $M$ and column $0$ of $W$ and thus
row $0$ of each submatrix $M_S$ and column $0$ of each submatrix $W_S$.
Since the resulting matrix $MW$ now has $n$ rows and $n$ columns,
the resulting subsets $S$ now have $n$ elements corresponding to $n$-arc
subgraphs of the matrix digraph.  Our procedure then is
to consider $n$ element subsets $S$ and
to work with the $(n + 1) \times n$ submatrices $M_S$ and
$n \times (n + 1)$
submatrices $W_S$ but then to strike row $0$ of $M_S$ and column $0$ of
$W_S$ before computing the determinants.

We also, without loss of generality, imagine that $M$ and $W$
are sorted by the index of their invertices.  In particular,
the columns of $M$ are sorted by invertex of the arc corresponding to
the column.  The rows of $W$ are then sorted by invertex of the arc
corresponding to the row.  There are no arcs into
vertex $0$; thus, the first $\ell_1$ columns of $M$ correspond to the $\ell_1$
arcs that have vertex $1$ as the invertex (and thus have a $1$ in row $1$).
The first $\ell_1$ rows of $W$ thus have entries (the values $w(e_k)$)
in column $1$.  The $\ell_1$ columns in $M$ are then followed by $\ell_2$
columns in $M$ that correspond to the $\ell_2$ arcs that have vertex $2$
as the invertex, and the $\ell_1$ rows in $W$ are followed by $\ell_2$
rows in $W$ with entries in column $2$.  This sorting proceeds until all arcs
are accounted for.

We now consider the submatrices $M_S$ and $W_S$.
$W_S$ is an
$n \times (n + 1)$ weight matrix whose rows correspond to the same arc as do
the columns in
$M_S$.  The first column is all zeros, but, because of our sorted arrangement
of the rows of $W$, there is one entry per row and the column number of
the non-zero element in each row is larger than or equal to that
in the previous row.

\begin{lemma}
If a subgraph in the matrix digraph consists of a set of
arcs $\{e_k\}$ with $k \in S$
and has one or more
vertices with indegree larger than one, then $det(W_S) = 0$.  Otherwise,
\begin{equation}
det(W_S) = \prod_{k \in S} w(e_k).
\label{eq:detws}
\end{equation}
\label{lemma:detws}
\end{lemma}

\begin{proof}
Consider a subgraph of the matrix digraph that consists of
a set of arcs $\{e_k\}$ with $k \in S$.
The arcs in the subgraph correspond to rows in $W_S$.
Each row of $W_S$ has a single non-zero element in the column corresponding
to the invertex of corresponding arc.
Suppose two rows in $W_S$ have
the same column number for their non-zero elements.  When column 0 of
$W_S$ is struck, there must be a
zero in the diagonal element of one of the rows.
This means that at least one
of the columns in $W_S$ (in addition to column $0$)
must contain all zeros and, after striking
column $0$ in the $n \times (n+1)$ version of $W_S$, $det(W_S) = 0$.
Thus,
no subgraph of the digraph contributes to $det( MW)$
if it contains
a vertex with indegree equal to two.  This holds {\em a fortiori} if 
the subgraph has a vertex with
indegree greater than two because, in such a case, there
will be more than one column in $W_S$ (after striking column $0$) containing
all zeros.

Only subgraphs of the matrix digraph that have indegree equal to
one for each vertex other than $0$ contribute to $det(MW)$.  For such
subgraphs, because of the sorted arrangement of the arcs, the contributing
submatrix $W_S$ will be diagonal.  The determinant of a diagonal
matrix is the product of its diagonal elements; hence, Eq. (\ref{eq:detws}).
\end{proof}

We now consider the incidence matrices $M_S$.
\begin{lemma}
If a subgraph in the matrix digraph consists of a set of
arcs $\{e_k\}$ with $k \in S$ and
contains a cycle,
then $det(M_S) = 0$.  Otherwise, $det(M_S) = 1$.
\label{lemma:detms}
\end{lemma}

\begin{proof}
Before striking row $0$, $M_S$ is an $(n + 1) \times n$ incidence matrix.
It corresponds to an $n$-arc subgraph of the full digraph.
The columns in $M_S$ correspond to a particular subset $\{e_k\}$ of
arcs in the matrix digraph such that $k \in S$.  By lemma \ref{lemma:detws},
any $M_S$ that
has one in more than one row may be excluded since it corresponds to
a subgraph with a vertex with indegree larger than unity.

We now consider
the remaining $M_S$ that may correspond to non-zero contributions to
Eq. (\ref{eq:cauchy-binet}).  In such an $M_S$,
each arc either has vertex $0$ as a source
or does not.
Consider a column $k_1$ in $M_S$ that corresponds
to an arc $e_{k_1}$ such that $s(e_{k_1}) \neq 0$.  By lemma \ref{lemma:detws},
each vertex other than $0$ in the subgraph $S$ has indegree exactly
one. There thus must be another arc $e_{k_2}$ in the subgraph with 
$t(e_{k_2}) = s( e_{k_1} )$.  We add column $k_2$ corresponding to
arc $e_{k_2}$ to column $k_1$.  Column $k_1$
now has $-1$ in row $s(e_{k_2})$ and
$1$ in row $t(e_{k_1})$ unless $s(e_{k_2}) = t(e_{k_1})$, in which
case column $k_1$ now has all zeros and,
after striking row $0$ in $M_S$, $det(M_S) = 0$.
Because $s(e_{k_2}) = t(e_{k_1})$, arcs $k_1$ and $k_2$ form a two-arc
cycle; thus, the subgraph corresponding to subset $S$ must contain
no two-arc cycles to contribute to $det(MW)$.

The argument may be extended.
If $s(e_{k_2}) \neq 0$, we may repeat the above procedure by
finding the arc $e_{k_3}$ whose target is the source of $e_{k_2}$.  We add
column $k_3$ to column $k_1$.  Now column $k_1$ has $-1$ in row
$s(e_{k_3})$ and $1$ in row $t(e_{k_1})$.  This procedure is repeated
until column operations have converted
the column $k_1$ into one in which there is
$-1$ in row $0$ and $1$ in row $t(e_{k_1})$.  If at any stage of
the regression a cycle forms, the column $k_1$
will have all zeros and, after striking row $0$, $det(M_S) = 0$.
We repeat this procedure for all columns that correspond to
arcs whose source is not vertex $0$.  If no cycles appear,
the resulting incidence matrix will have $-1$ in each column of
row $0$ and, because of the sorted arrangement of the arcs, $1$
in the $(i+1,i)$ element.

The above regression procedure thus
produces a new incidence matrix $M_S'$ corresponding to a
graph that has a single arc from
vertex $0$ to each of the other vertices, if the subgraph is
acyclic.  Otherwise $det(M_S) = 0$.  In other words, the
subgraph will only contribute to $det(MW)$ if there is a
path (defined in the usual sense as a sequence of arcs that joins a sequence of distinct vertices) from vertex $0$ to each of the other vertices.  The new incidence matrix is that for an arborescence of the
graph in which all arcs have vertex $0$ as the source.
We call this graph the {\em root graph} of $S$.  In general,
an arc in this graph from vertex $i$ to vertex $j$ means that
there is a path from vertex $i$ to vertex $j$ in the parent graph.
In our particular case, we see that only subgraphs that have
a path from vertex $0$ to each of the other vertices $i$
contribute to the overall determinant.

If we now strike row $0$ from the $(n + 1) \times n$ version of
$M_S'$, we are left with an $n \times n$ identity matrix.
Thus, $det(M_S') = 1$.  Because addition
of columns in a matrix leaves the determinant of the matrix unchanged,
$det(M_S) = det(M_S') = 1$.
\end{proof}

We now prove our version of the matrix-tree theorem.

\begin{theorem}
Consider the $n \times n$ matrix $A$ in Eq. (\ref{eq:aij}).
\begin{equation}
det(A) =
\sum_{S} \prod_{k \in S} w(e_k)
\label{eq:detDAsum}
\end{equation}
where $S$ is a subset of arc labels that correspond to a subset of arcs
in the matrix digraph of $A$ that form an arborescence.
\label{theorem:detDA}
\end{theorem}

\begin{proof}
Consider a subset $S$ in Eq. (\ref{eq:cauchy-binet}).
If the $n$ arcs $e_k$ for $k \in S$
form an acyclic subgraph of the
matrix digraph (that includes vertex $0$) with
indegree equal to zero for vertex $0$ and indegree equal
to one for all other vertices, then
by lemmas \ref{lemma:detws} and \ref{lemma:detms},
\begin{equation}
det(M_S) det(W_S) = \prod_{k \in S} w(e_k)
\label{eq:detmsws}
\end{equation}
Otherwise, $det(M_S) det(W_S) = 0$.
The $n$ arc subset $\{e_k\}$ for
$k \in S$ constitutes an arborescence of the digraph rooted at
vertex $0$.  By Eq. (\ref{eq:cauchy-binet}),
Eq. (\ref{eq:detDAMW}) holds, with subsets $S$ restricted to
those corresponding to arborescences in the matrix digraph.
\end{proof}

As an example of Theorem \ref{theorem:detDA}, consider the determinant of the matrix in Eq. (\ref{eq:A}) and the corresponding matrix digraph in Fig. \ref{fig:3digraph}.  The sixteen arborescences are shown in Fig. \ref{fig:branchings}, and the sum of the weights of these branching is 42, in agreement with Eq. (\ref{eq:example_det}).

\begin{figure}
    \centering
    \begin{subfigure}[t]{0.2\textwidth}
        \centering
        \includegraphics[width=\linewidth]{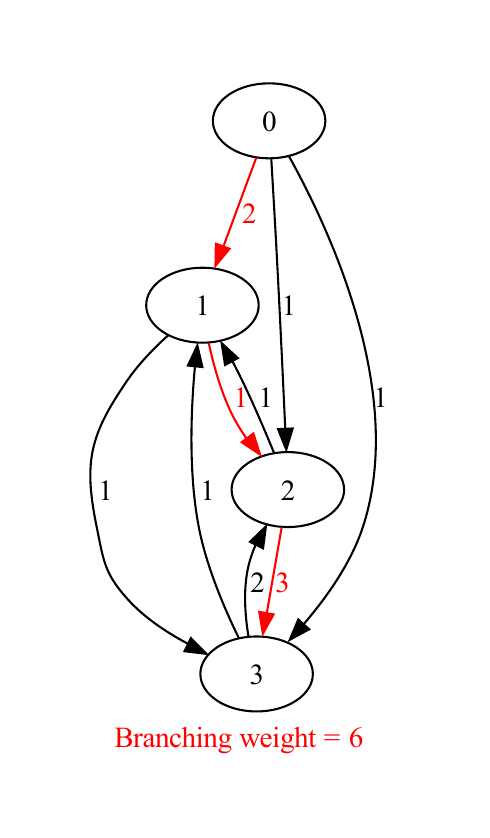} \subcaption{}
    \end{subfigure}
    \hfill
    \begin{subfigure}[t]{0.2\textwidth}
        \centering
        \includegraphics[width=\linewidth]{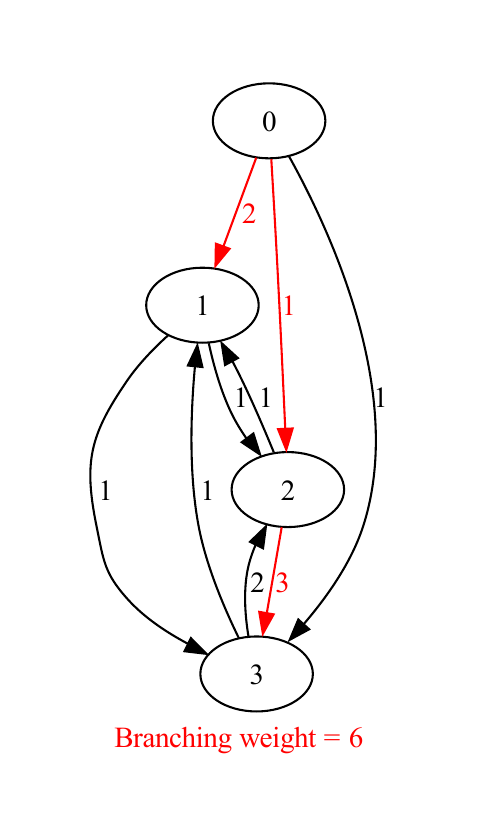} 
        \subcaption{}
    \end{subfigure}
    \hfill
    \begin{subfigure}[t]{0.2\textwidth}
        \centering
        \includegraphics[width=\linewidth]{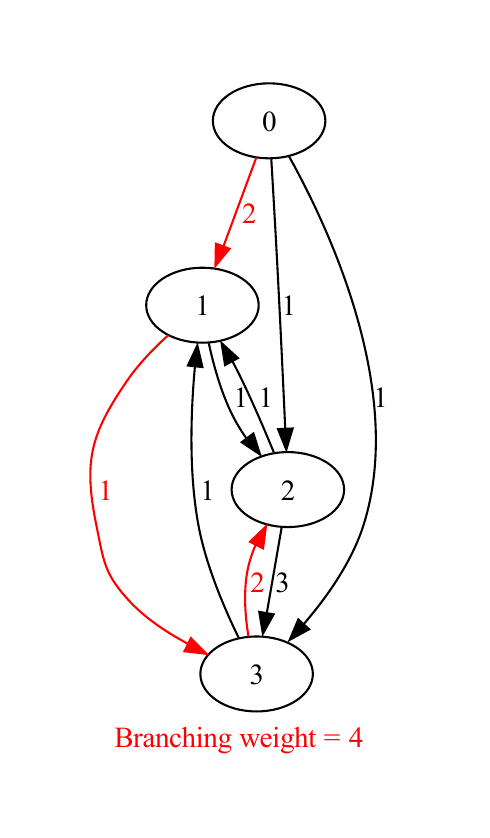} 
        \subcaption{}
    \end{subfigure}
    \hfill
    \begin{subfigure}[t]{0.2\textwidth}
        \centering
        \includegraphics[width=\linewidth]{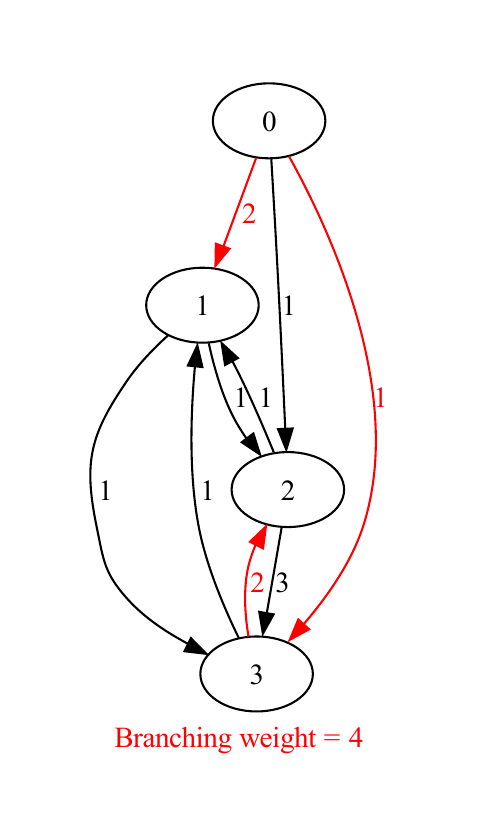} 
        \subcaption{}
    \end{subfigure}
    \begin{subfigure}[t]{0.2\textwidth}
        \centering
        \includegraphics[width=\linewidth]{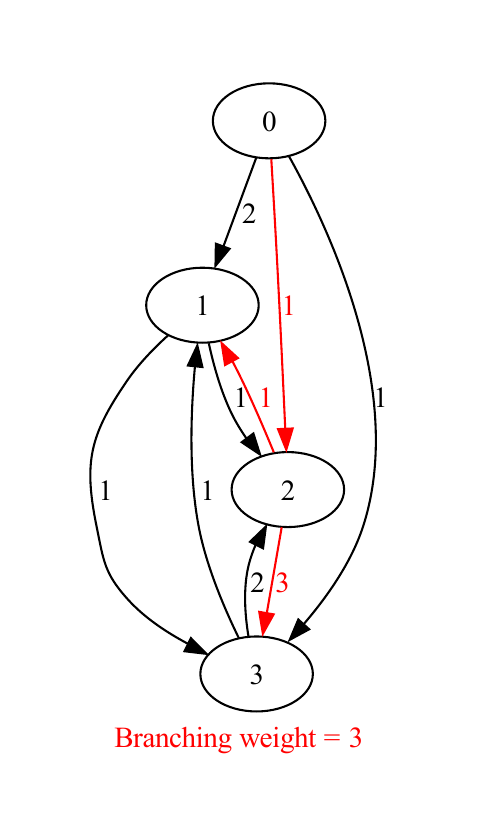} \subcaption{}
    \end{subfigure}
    \hfill
    \begin{subfigure}[t]{0.2\textwidth}
        \centering
        \includegraphics[width=\linewidth]{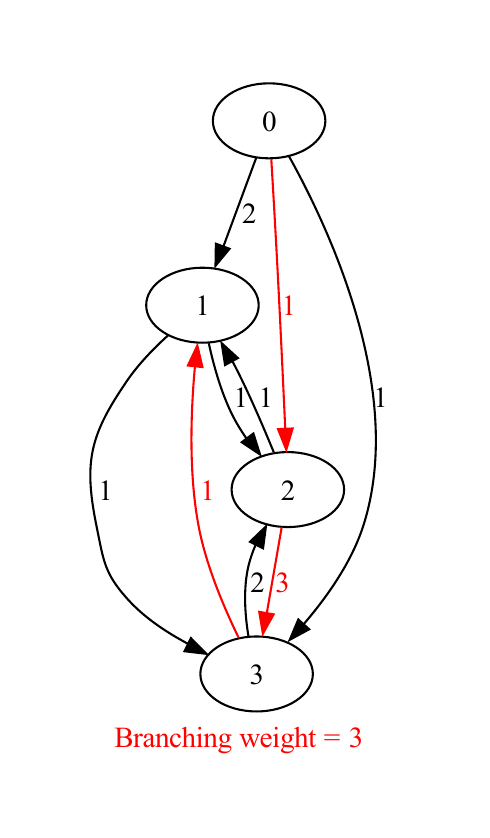} 
        \subcaption{}
    \end{subfigure}
    \hfill
    \begin{subfigure}[t]{0.2\textwidth}
        \centering
        \includegraphics[width=\linewidth]{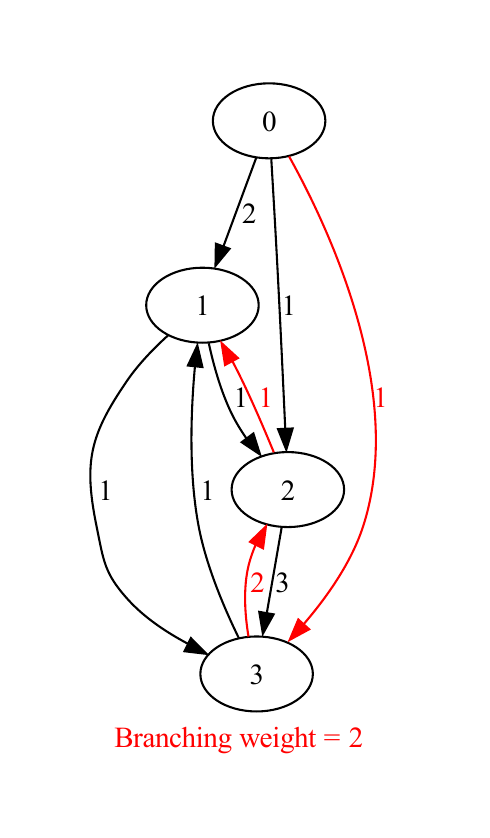} 
        \subcaption{}
    \end{subfigure}
    \hfill
    \begin{subfigure}[t]{0.2\textwidth}
        \centering
        \includegraphics[width=\linewidth]{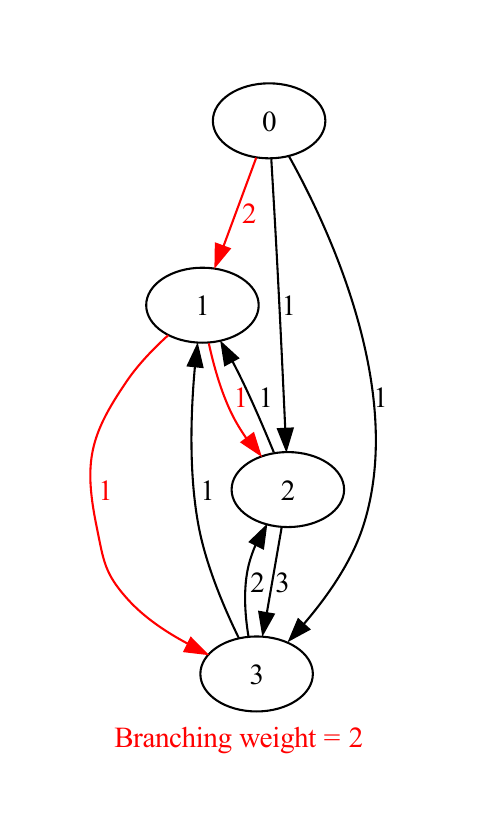} 
        \subcaption{}
    \end{subfigure}
        \begin{subfigure}[t]{0.2\textwidth}
        \centering
        \includegraphics[width=\linewidth]{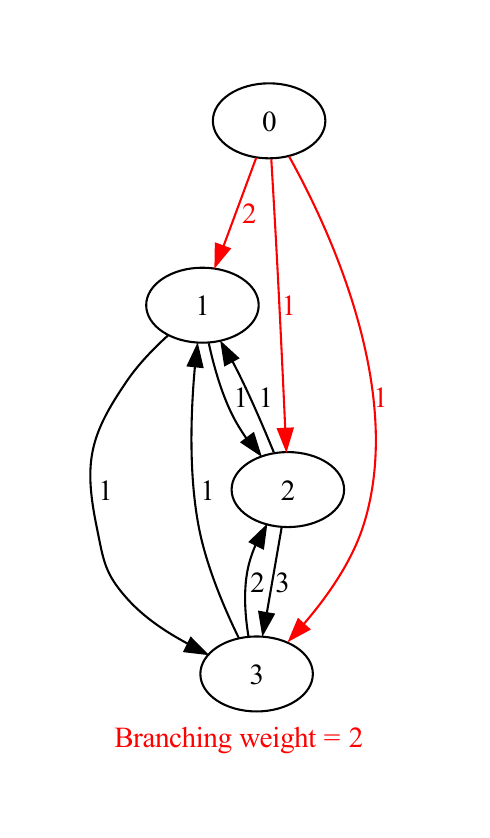} \subcaption{}
    \end{subfigure}
    \hfill
    \begin{subfigure}[t]{0.2\textwidth}
        \centering
        \includegraphics[width=\linewidth]{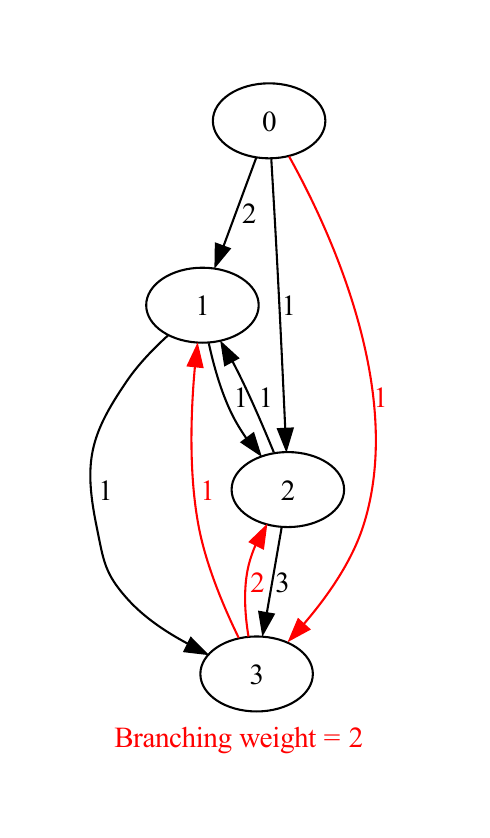} 
        \subcaption{}
    \end{subfigure}
    \hfill
    \begin{subfigure}[t]{0.2\textwidth}
        \centering
        \includegraphics[width=\linewidth]{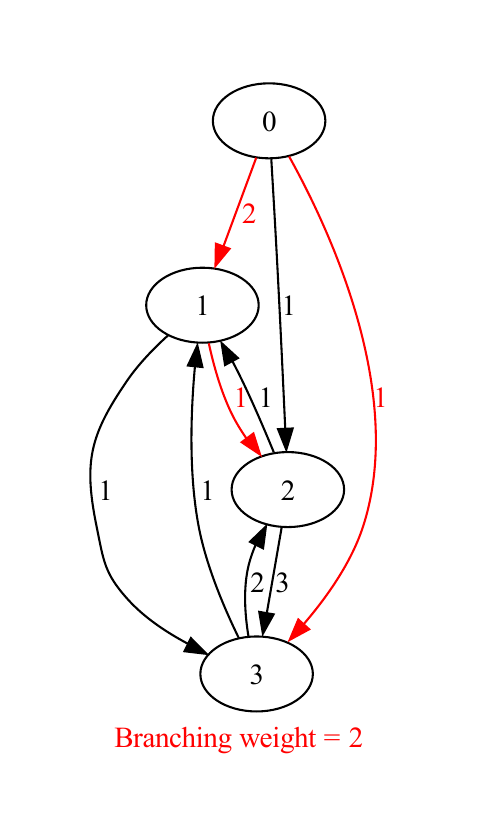} 
        \subcaption{}
    \end{subfigure}
    \hfill
    \begin{subfigure}[t]{0.2\textwidth}
        \centering
        \includegraphics[width=\linewidth]{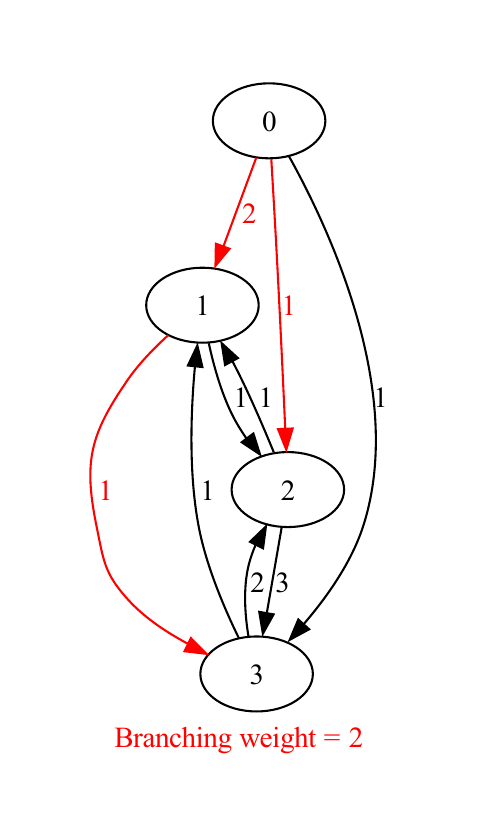} 
        \subcaption{}
    \end{subfigure}
    \begin{subfigure}[t]{0.2\textwidth}
        \centering
        \includegraphics[width=\linewidth]{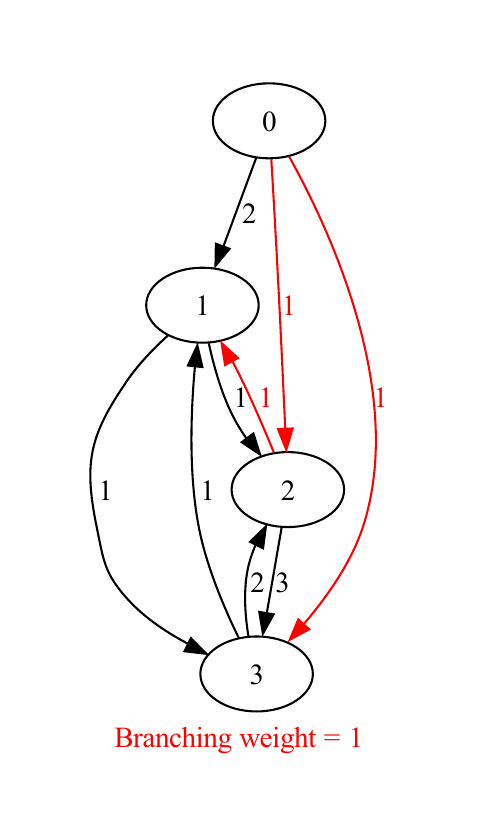} \subcaption{}
    \end{subfigure}
    \hfill
    \begin{subfigure}[t]{0.2\textwidth}
        \centering
        \includegraphics[width=\linewidth]{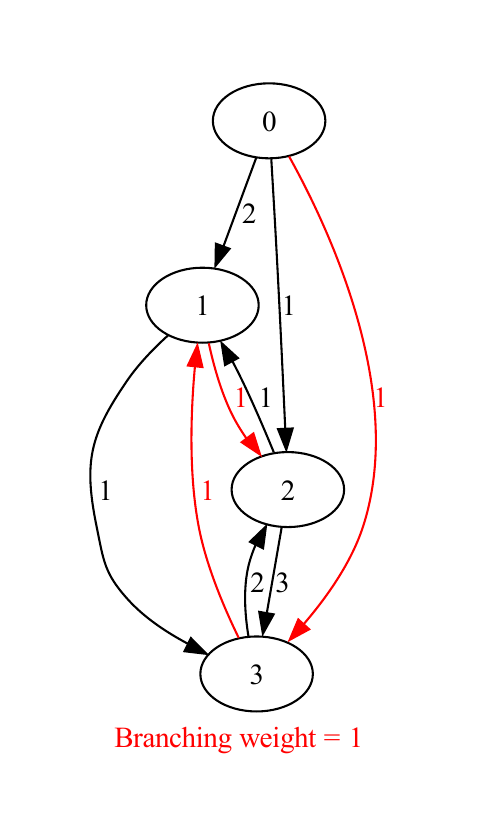} 
        \subcaption{}
    \end{subfigure}
    \hfill
    \begin{subfigure}[t]{0.2\textwidth}
        \centering
        \includegraphics[width=\linewidth]{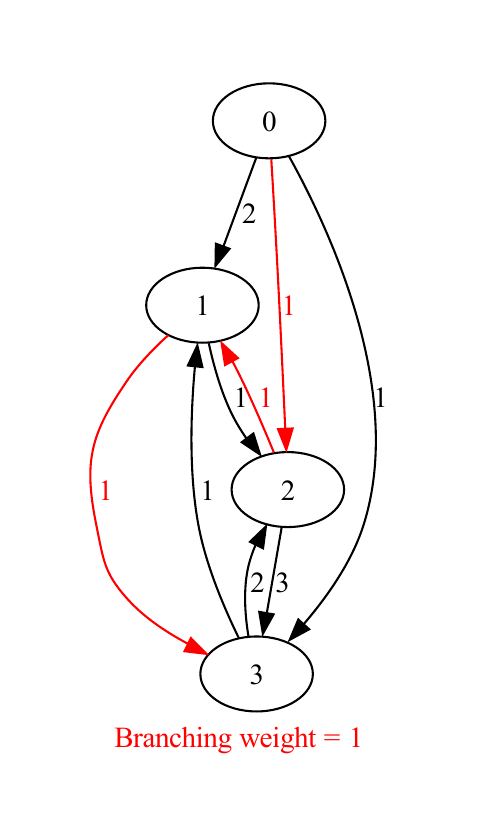} 
        \subcaption{}
    \end{subfigure}
    \hfill
    \begin{subfigure}[t]{0.2\textwidth}
        \centering
        \includegraphics[width=\linewidth]{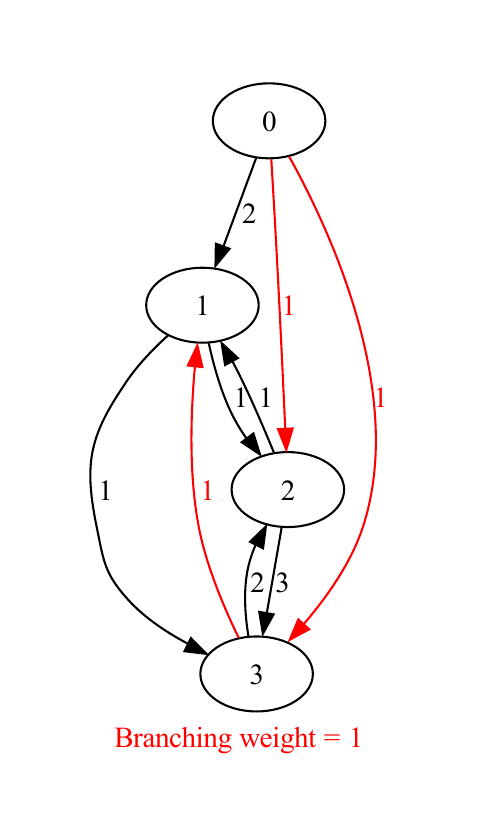} 
        \subcaption{}
    \end{subfigure}
\caption{The sixteen branchings for the matrix digraph in Fig. \ref{fig:3digraph}.}
\label{fig:branchings}
\end{figure}

\section {Reduced-Matrix-Tree Theorem}
\label{sec:reduced_matrices}

Chen \cite{chen2012applied} and Chaiken \cite{chaikenSeth} provide a generalization of the matrix-tree theorem to additional minors of the original zero-column-sum matrix.  We provide our own treatment of this problem but instead consider {\em reduced} versions of a general matrix.

\begin{defn}
Let $P = \{p_1, p_2, ..., p_m\}$ and
$Q = \{q_1, q_2, ..., q_m\}$ be two nonempty subsets
of the set of integers $\{1, 2, ..., n\}$ such that $m \leq n$.
Let $M$ be an $n \times n$ matrix with elements $M_{i,j}$.
The {\it reduced matrix} $M^{P, Q}$ is the matrix derived from $M$ with
matrix elements $[M^{P, Q}]_{i, j} = \delta_{i,p_k} \delta_{j,q_k}$ for
$j \in Q$ and
$[M^{P, Q}]_{i, j} = M_{i, j}$ for $j \notin Q$.
The matrix $M^{P, Q}$ is thus the matrix $M$ with elements in
column $q_k \in $ in $Q$ replaced by zeros except for the row given
by $p_k$, which has the value 1.
\label{defn:reducedA}
\end{defn}

The set $Q$ may contain elements $q_k \in P$ that are not
in the same position in $Q$ as in $P$; that is, $q_k = p_j$ with $k \ne j$.
One may pairwise exchange elements of $Q$ to form a set
$\tilde{Q} = \{\tilde{q}_1, \tilde{q}_2, ..., \tilde{q}_m\}$ such that
$\tilde{q}_k = p_k$ for each element $q_k \in P$.
Let $N_Q$ be
the minimum number of pairwise exchanges needed to convert $Q$ into $\tilde{Q}$.

Now it is possible to compute the determinant of a reduced matrix
in terms of weighted sums of arborescences.  In what follows, vertices of
an $n \times n$
matrix digraph will be referred to by their label (an element of the
set $\{1, 2, ..., n\}$ and the numbers $p_k \in P$ and
$\tilde{q}_k \in \tilde{Q}$),
and the root matrix of the digraph will be the vertex $0$.

\begin{theorem}
Consider the matrix $A$ in Eq. (\ref{eq:aij}).
\begin{equation}
det(A^{P, Q}) = \left(-1\right)^{N_Q} \sum_{B: \{(\tilde{q_k})_1 \to p_j\}}
\epsilon(B) W(B)
\label{eq:detapq}
\end{equation}
where $B$ is an arborescence in the matrix digraph of
$A$ except that $\{(\tilde{q_k})_1 \to p_j\}$
indicates $B$ is rooted at each vertex
$\tilde{q}_k \in \tilde{Q}$ with weight 1 and with a path to some
vertex $p_j \in P$, $W(B)$ is the weight of arborescence $B$ such that $W(B) = \prod_{e \in B} w(e)$, and where $\epsilon(B)$ is a factor $\pm 1$, depending
on the cycles in a subgraph that can be derived from $B$.
\label{theorem:reduced}
\end{theorem}

\begin{proof}
Rearranging the set $Q$ to $\tilde{Q}$ requires $N_Q$
exchanges of ${q_k}'s$.
Since each exchange of $q_k$'s
corresponds to an exchange of columns in $A^{P, Q}$, which, in turn, changes
the sign of the resulting determinant,
\begin{equation}
det(A^{P, Q}) = \left(-1\right)^{N_Q} det(A^{P, \tilde{Q}})
\label{eq:detapqp}
\end{equation}

$A^{P, \tilde{Q}}$ has all zeros in each column $\tilde{q}_k \in \tilde{Q}$
except a 1 in row
$p_k$ for column $\tilde{q}_k$.
If $p_k = \tilde{q}_k$, the matrix digraph for $A^{P, \tilde{Q}}$
has an arc $(0, \tilde{q}_k)$ with weight 1 as the only in arc
to the vertex $\tilde{q}_k$.  Since $p_k = \tilde{q}_k$, $p_k$ and
$\tilde{q}_k$ are the same vertex.  Nevertheless, in this case we consider
there to be a path $\tilde{q}_k \to p_k$.
Let $R \subseteq \tilde{Q}$ be the subset of vertices for which
$\tilde{q}_k = p_k$ for $\tilde{q}_k \in \tilde{Q}$ and $p_k \in P$.

If $p_k \ne \tilde{q}_k$, the vertex $\tilde{q}_k$
will have two in arcs in the $A^{P, \tilde{Q}}$ matrix digraph.
The first will be an arc $(p_k, \tilde{q}_k)$ with weight -1.
Since the diagonal element in $A^{P, \tilde{Q}}$ 
in column $\tilde{q}_k$ and row $\tilde{q}_k$ for this case will be zero,
the second arc is $(0, \tilde{q}_k)$ with weight 1 to cancel out the
weight of the first in arc.
Let $S \subseteq \tilde{Q}$ be the subset of vertices for which
$\tilde{q}_k \ne p_k$ for $\tilde{q}_k \in \tilde{Q}$ and $p_k \in P$.

All arborescences in the matrix digraph of $A^{P, \tilde{Q}}$ must have
either only the arc $(0, \tilde{q}_k)$ if $\tilde{q}_k \in R$ or either
the arc $(0, \tilde{q}_k)$ or $(p_k, \tilde{q}_k)$
if $\tilde{q}_k \in S$.  This means that all arborescences
in the matrix digraph of $A^{P, \tilde{Q}}$ can be derived from an arborescence
rooted at each $\tilde{q}_k \in \tilde{Q}$ by replacing one or more
of the arcs in $(0, \tilde{q}_k)$ for $\tilde{q}_k \in S$ by their
complement arcs $(p_k, \tilde{q}_k)$, as long as the resulting subgraph
does not contain one or more cycles, in which case the subgraph would not
be an arborescence.

Consider an arborescence $B$ in the matrix digraph of $A^{P, \tilde{Q}}$ that
is rooted at each vertex $\tilde{q}_k \in \tilde{Q}$.  Suppose further that
for one of these root vertices $\tilde{q}_j$ there is no path to any
$p_\ell \in P$.  The vertex $p_j$ is thus part of a path from some other
rooted vertex of $B$ to $p_j$.  There is one and only one arborescence
 $B'$ in the
sum over arborescences of the matrix digraph of $A^{P, \tilde{Q}}$ that is
identical to $B$ except that it has the arc $(p_j, \tilde{q}_j)$ with weight
-1 in place of the arc $(0, \tilde{q}_j)$ with weight 1; thus,
$W(B') = -W(B)$, and these arborescences will cancel in the sum over
arborescences.
Since $\vert \tilde{Q} \vert = \vert P \vert$, for an arborescenc $B$ that
is rooted at each vertex $\tilde{q}_k \in \tilde{Q}$, there must
be a path to one and only one vertex $p_j \in P$ from $\tilde{q}_k$.
Recall that we consider
there to be a path $\tilde{q}_k \to p_k$ when $\tilde{q}_k = p_k$.

Now consider an arborescence
$B$ rooted at each vertex $\tilde{q}_k \in \tilde{Q}$
and with a path to one and only one vertex $p_j \in P$.
Consider the subgraphs derived from {\em parent} arborescence
$B$ by replacing one or more
arcs $(0, \tilde{q}_k)$ for $\tilde{q}_k \in S$ by $(p_k, \tilde{q}_k)$.
These subgraphs are arborescences
as long as the replacements do not lead to a cycle or cycles.
Such a cycle $C$ would
be $\tilde{q}_\ell \to ... \to p_j \to \tilde{q}_j \to ... \to p_\ell
\to \tilde{q}_\ell$
and would
result from replacing all $(0, \tilde{q}_k)$ arcs with $(p_k, \tilde{q}_k)$
arcs for each $\tilde{q}_k$ in the cycle.
Let the number of vertices $\tilde{q}_k$ in $C$ be $N_C$.
The sum over arborescence weights derived from parent arborescence
$B$ would be
\begin{equation}
\Sigma(B) = W(B) \prod_{C \in \{C\}_B}
\left(\sum_{r = 0}^{N_C - 1} \binom{N_C}{k}
\left(-1\right)^r 1^{N_C - i}\right)
\label{eq:SigmaB}
\end{equation}
where $C$ is any cycle that can be present among a subset of $\tilde{q}_k$
vertices in $B$ upon replacement of all in arcs to the vertices with
arcs $(p_k, \tilde{q}_k)$ and where $\{C\}_B$ is the set of all such
cycles that can be derived from $B$.  $\Sigma(B)$ may be written
\begin{equation}
\Sigma(B) = W(B) \prod_{C \in \{C\}_B}\left([1 - 1]^{N_C} - \left(-1\right)^{N_C} \right)
= W(B) \prod_{C \in \{C\}_B}\left(-1\right)^{N_C - 1}
\label{eq:SigmaB2}
\end{equation}
Since all arborescences
in the matrix digraph of $A^{P, \tilde{Q}}$ can be derived
from arborescence $B$ rooted at each vertex $\tilde{q}_k \in \tilde{Q}$ with
a path to one and only one $p_j \in P$,
\begin{equation}
\det\left(A^{P, \tilde{Q}}\right) = \sum_{B:\{(q_k)_1 \to p_j\}} \Sigma(B)
\label{eq:detaqp2}
\end{equation}
The arborescences $B$ include only arcs from the digraph of matrix $A$, except
for the root arcs to vertices $\tilde{q}_k$, which have weight 1, so
combination of Eqs. (\ref{eq:detapqp}), (\ref{eq:SigmaB2}),
and (\ref{eq:detaqp2}) with
\begin{equation}
\epsilon(B) = \prod_{C \in \{C\}_B} \left(-1\right)^{\vert C \vert - 1}
\label{eq:epsilonB}
\end{equation}
completes the proof.
\end{proof}

\begin{remark}
Since the weights of the arcs $(0, \tilde{q}_k)$ in the arborescences $B$
in Eq. (\ref{eq:detapqp}) are all 1, those arc weights do not contribute
to $W(B)$.  This means that those arcs can be removed without changing
the sum over arborescence weights and thus that
each arborescence $B$ in Eq. (\ref{eq:detapqp})
can be viewed as the union of an arborescence with root 0
and rooted at vertices not in $\tilde{Q}$ and
a forest of arborescences each with root $\tilde{q}_k \in \tilde{Q}$ and
containing only one vertex $p_j \in P$.  The factor $\epsilon(B)$ can
be computed in this case by computing cycles by connecting the vertex $p_j$
in the arborescence with root $\tilde{q}_k$ to the arborescence with root
$\tilde{q}_j$ via the arc $(p_j, \tilde{q}_j)$ with weight -1.  Through
such operations, a cycle $C$ results when $\vert C \vert$ separate arborescences have
been combined such that no vertex in the resulting subgraph has indegree
zero.  From the set of cycles $\{C\}_B$ one derives in this way from
parent $B$, Eq. (\ref{eq:epsilonB}) yields  $\epsilon(B)$.  Of course, if the original matrix $A$ is a zero-column-sum matrix ($v_{ii} = 0$ for all $i$), then there is no arborescence rooted at root vertex 0, and the result is a sum over a forest of directed trees, each rooted at a vertex in $Q$, as discussed in Chaiken \cite{chaikenSeth} and Moon \cite{MOON1994163}, except that those treatments include a factor $(-1)^{\sum_{k=1}^m (p_k + q_k)}$ to account for the fact that the minors have rows in $P$ and columns in $Q$ struck relative to our reduced matrix.
\label{remark:minor}
\end{remark}

\begin{remark}
A modified reduced matrix may have more than one 1 in a column in which the other entries are zero.  Such a case may be handled by considering that matrix as the sum of one or more matrices.  The determinant will be the sum of determinants of matrices, each a reduced matrix as we have defined it, derived from different permutations of the rows or columns of the matrices that sum to give the modified reduced matrix.
\end{remark}

\begin{remark}
    By Remark \ref{remark:minor}, the $j,i$ cofactor $M_{ji}$ of a matrix $A$ is the determinant of the reduced matrix of $A$ with zeros in column $i$ except a 1 in row $j$; that is, $M_{ji} = det(A^{j,i})$.  The cofactor matrix $M$ for matrix $A$ is the matrix with elements $M_{ji}$, the cofactors of $A$.  The $i,j$ element of $A^{-1}$, the inverse of $A$, if $A$ is invertible, is the adjugate matrix $adj(A)$ of $A$ divided by the determinant; that is, $A^{-1} = adj(A) / det(A)$.  Since the adjugate matrix is the transpose of the cofactor matrix, $(A^{-1})_{ij} = det(A^{j,i})/det(A)$.  Then, by Theorems \ref{theorem:detDA} and \ref{theorem:reduced}, the $i,j$ element of the inverse of $A$ is the sum of the weights of all arborescences in the matrix digraph of $A$ rooted at $i$ with a path from $i$ to $j$ divided by the sum of all arborescence weights.
    \label{remark:inverse}
\end{remark}

The inverse of the matrix in Eq. (\ref{eq:A}) is
\begin{equation}
    A^{-1} = \frac{1}{42}.\begin{pmatrix}
14 & 7 & 7\\
 8 & 19 & 13 \\
6  & 9 & 15
\end{pmatrix}
\label{eq:Ai}
\end{equation}

We may understand this result in light of Remark \ref{remark:inverse}.
To get the term corresponding to $A^{-1}_{11}$, we note that this element can be obtained from the summation of all arborescences that are rooted at vertex $1$ (any arborescence rooted at 1 has a path to 1). These arborescences are Figs. \ref{fig:branchings}(a), \ref{fig:branchings}(b), \ref{fig:branchings}(c), \ref{fig:branchings}(d), \ref{fig:branchings}(h), \ref{fig:branchings}(i), \ref{fig:branchings}(k) and \ref{fig:branchings}(l).  The sum of the weights of these aborescences is 28.  Recall, however, that the arc $(0, 1)$ in these arborescences must be replaced by one with weight 1.  Since the original $(0, 1)$ arc had weight 2, this halves the sum to give 14, as expected.  To compute $A^{-1}_{11}$, we divide this result by 42, the determinant of $A$, which we computed from the sum of all arborescence weights in Fig. \ref{fig:branchings}.

Similarly, we can find the element $A^{-1}_{32}$ in Eq. (\ref{eq:A}). This arises from the sum over weights of arborescences that are rooted at vertex $3$ with a path from vertex $3$ to vertex $2$ as shown in figure \ref{fig:branchings}(d),  \ref{fig:branchings}(g), \ref{fig:branchings}(j) and \ref{fig:branchings}(n).  The sum of the weights of these arborescences is 9, as expected, since the weight of the $(0, 3)$ arc is already 1.

Another example that demonstrates the aspects of the reduced-matrix tree theorem is shown below. Let us consider a $4 \times 4$ matrix  

\begin{equation}
    C = \begin{pmatrix}
5.71 & -1.61 & -1.45 & -1.34\\
-1.36 & 6.82 & -2.15 & -1.78\\
-1.52 & -2.43 & 10.63 & -2.91\\
-2.91 & -1.27 & -3.69 & 7.33
\end{pmatrix}
\label{eq:Mi}
\end{equation}

We can find the reduced matrix $C^{\{1,2,3\},\{2,1,4\}}$.
\begin{equation}
    C^{\{1,2,3\},\{2,1,4\}} = \begin{pmatrix}
0 & 1 & -1.45 & 0\\
1 & 0 & -2.15 & 0\\
0 & 0 & 10.63 & 1\\
0 & 0 & -3.69 & 0
\end{pmatrix}
\label{eq:M_cof}
\end{equation}
Here $P = \{1,2,3\}$ and $Q = \{2,1,4\}$.  We may swap columns 1 and 2 to get ${\tilde Q} = \{1,2,4\}$.  
\begin{equation}
    C^{\{1,2,3\},\{1,2,4\}} = \begin{pmatrix}
1 & 0 & -1.45 & 0\\
0 & 1 & -2.15 & 0\\
0 & 0 & 10.63 & 1\\
0 & 0 & -3.69 & 0
\end{pmatrix}
\label{eq:M_cof2}
\end{equation}
Then $det(C^{\{1,2,3\}, \{2,1,4\}}) = -det(C^{\{1,2,3\},\{1,2,4\}})$.  The seven arborescences in the matrix digraph corresponding to $C^{\{1,2,3\},\{1,2,4\}}$ are shown in Fig. \ref{fig:branchings2}.  The weight of any arborescence without a path from vertex 4 to vertex 3 is explicitly canceled by the weight of another such arborescence.  The only surviving arborescence is the one rooted at vertices 1, 2, and 4 and with a path from vertex 4 to vertex 3, as required by Theorem \ref{theorem:reduced}.  This has weight 3.69, so, by multiplying by -1 from the column swapping, $det(C^{\{1,2,3\},\{2,1,4\}}) = -3.69$.

\begin{figure}
    \centering
    \begin{subfigure}[t]{0.3\textwidth}
        \centering
        \includegraphics[width=\linewidth]{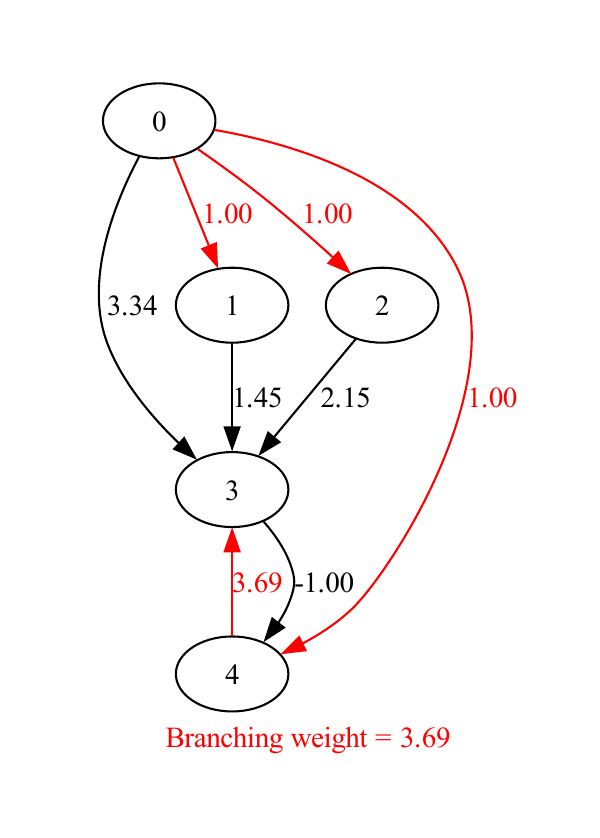} \subcaption{}
    \end{subfigure}
    \hfill
    \begin{subfigure}[t]{0.3\textwidth}
        \centering
        \includegraphics[width=\linewidth]{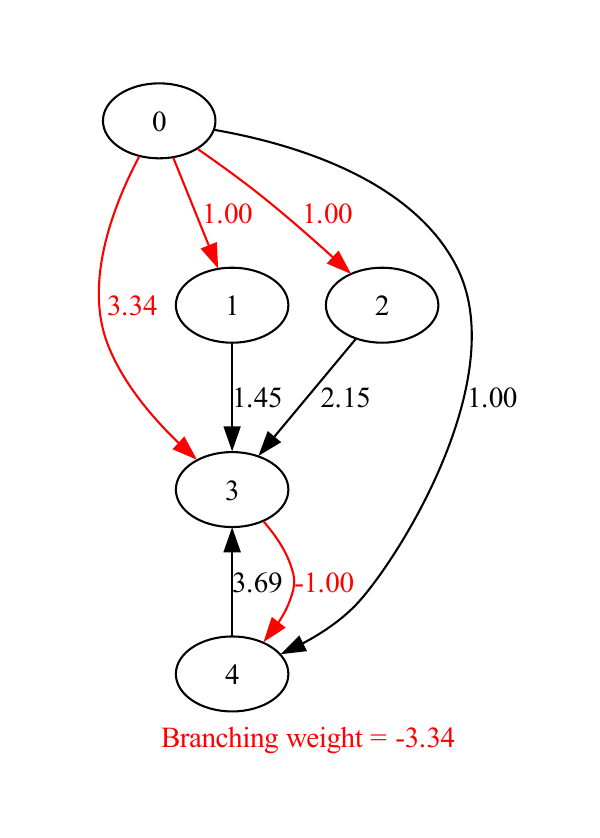} 
        \subcaption{}
    \end{subfigure}
    \hfill
    \begin{subfigure}[t]{0.3\textwidth}
        \centering
        \includegraphics[width=\linewidth]{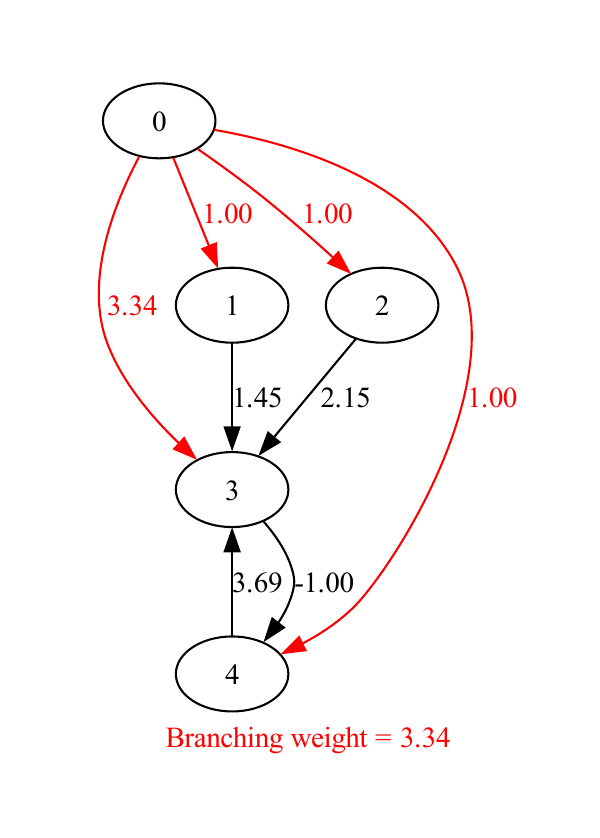} 
        \subcaption{}
    \end{subfigure}
    \hfill
    \begin{subfigure}[t]{0.3\textwidth}
        \centering
        \includegraphics[width=\linewidth]{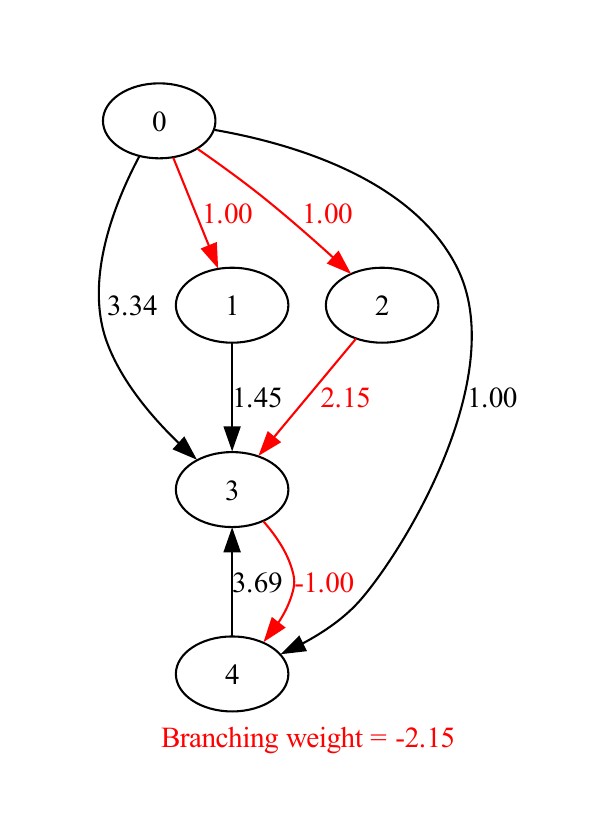} 
        \subcaption{}
    \end{subfigure}
        \begin{subfigure}[t]{0.3\textwidth}
        \centering
        \includegraphics[width=\linewidth]{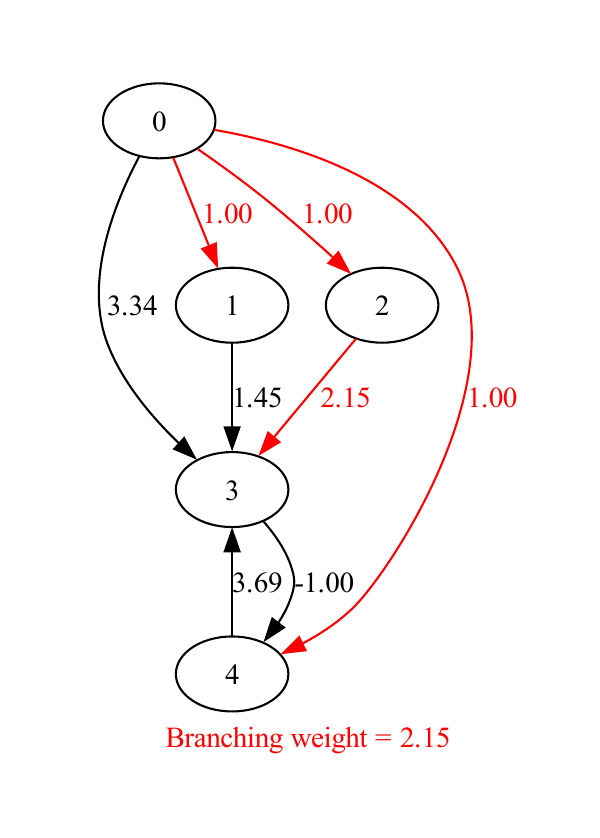} 
        \subcaption{}
    \end{subfigure}
        \begin{subfigure}[t]{0.3\textwidth}
        \centering
        \includegraphics[width=\linewidth]{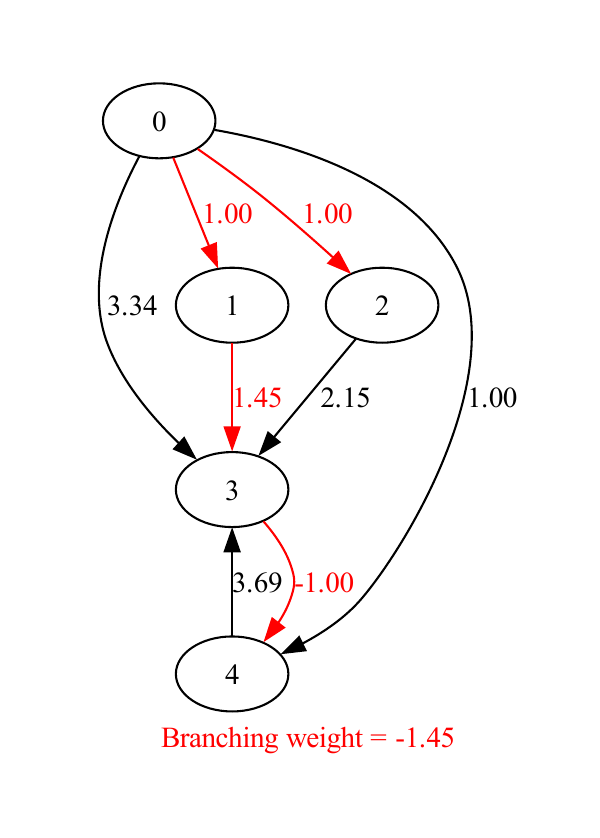} 
        \subcaption{}
    \end{subfigure}
        \begin{subfigure}[t]{0.3\textwidth}
        \centering
        \includegraphics[width=\linewidth]{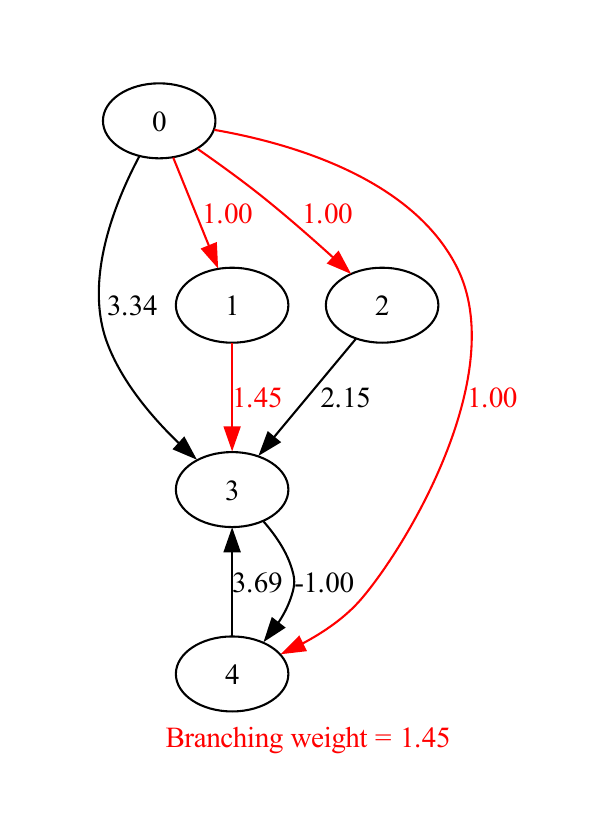} 
        \subcaption{}
    \end{subfigure}
\caption{Arborescences in the calculation of the determinant of the matrix in Eq. (\ref{eq:M_cof2}).}
\label{fig:branchings2}
\end{figure}

\section{An Illustrative Application} \label{sec:application}

In this section, we illustrate a practical use of Theorems \ref{theorem:detDA} and \ref{theorem:reduced} by applying them to the time-evolution of probabilities in a system with $n$ discrete states.  An example would be the evolution of population probabilities among excited levels in an ensemble of identical atoms.  The Theorems provide, in principle, a way of computing the system's evolution but, perhaps more importantly, a way of interpreting that evolution.

We take that the probability of the system to be in state $i$ is given by $X_i$.  We consider the rate of transition from state $i$ to state $j$ to be independent of any of the $X$'s and to be given by $\lambda_{ij}$, with $\lambda_{ij} \ge 0$ for all $i,j$.  The rate of change of probabilities is then given by
\begin{equation}
\frac{dX}{dt} = -\Lambda X
\label{eq:dXdt}
\end{equation}
where $X$ is the vector of probabilities $X_i$ and the elements of the matrix $\Lambda$ are given by
\begin{equation}
\Lambda_{ij} = 
\begin{cases}
-\lambda_{ji}, & i\ne j, 1 \leq i, j \leq n\\
 \sum_{k = 1, k\ne i}^n \lambda_{jk}, & i = j, 1 \leq i \leq n
\end{cases}
\label{eq:Lambda_elements}
\end{equation}

\begin{remark}
The matrix digraph $\Gamma_\Lambda$ corresponding to $\Lambda$ has arc $(i,j)$ between vertex $i$ and vertex $j$ with weight $\lambda_{ji}$.  Because the column sum for each column in $\Lambda$ is zero, there are no arcs $(0, i)$, that is, no arcs from the root vertex $0$ to any other vertex; therefore, $det(\Lambda) = 0$ and $\Lambda$ is singular.
\label{remark:Lambda}
\end{remark}

Matrix exponentiation techniques are available to solve Eq. (\ref{eq:dXdt}) when the rates $\lambda_{ij}$ are constant in time (e. g., \cite{doi:10.1137/09074721X}), but here we consider implicit Euler finite differencing in time.  This is especially useful if the $\lambda_{ij}$'s are time-dependent (for example, the temperature of the system, upon which the rates $\lambda_{ij}$ depend, is varying in time).  In such a case, Eq. (\ref{eq:dXdt}) becomes
\begin{equation}
    \frac{X(t + \Delta t) - X(t)}{\Delta t} = -\Lambda X(t + \Delta t)
    \label{eq:finite}
\end{equation}
In this approach, the system evolves over a finite timestep $\Delta t$ from $t$ to $t + \Delta t$.  The finite derivative is computed at $t + \Delta t$ in keeping with the implicit approach, which is appropriate for stiff systems of coupled equations (e.g., \cite{2006NuPhA.777..188H}).  This yields the matrix equation
\begin{equation}
    \left(I + \Lambda \Delta t\right)X(t + \Delta t) = X(t)
    \label{eq:finite_mat}
\end{equation}
where $I$ is the $n \times n$ identity matrix.
We let
\begin{equation}
    {\cal M} = I + \Lambda \Delta t
    \label{eq:cal_M}
\end{equation}
The solution to Eq. (\ref{eq:finite_mat}) is then
\begin{equation}
    X(t + \Delta t) = {\cal M}^{-1} X(t)
    \label{eq:finite_sol}
\end{equation} 
The probability $X_i(t + \Delta t)$ is
\begin{equation}
    X_i(t + \Delta t) = \sum_j {\cal M}_{ij}^{-1} X_j(t)
    \label{eq:X_tdt}
\end{equation}

Suppose $\Gamma_{\cal M}$ is the matrix digraph corresponding to ${\cal M}$.  It has arcs $(i,j)$ with weight $\lambda_{ji} \Delta t$ for $i, j \ne 0$.  In addition, there is an arc $(0, i)$ to each vertex $i$ with weight 1 arising from the identity matrix part of ${\cal M}$.  By Remark \ref{remark:inverse}, ${\cal M}_{ij}^{-1}$ is given by the sum of weights of all arborescences in $\Gamma_{\cal M}$ with arc $(0, i)$ and a path from vertex $i$ to vertex $j$ divided by the sum of weights of all arborescences in $\Gamma_{\cal M}$.  By taking account of Remark \ref{remark:minor}, we find
\begin{equation}
    {\cal M}_{ij}^{-1} = \frac{\delta_{ij} + \sum_{m=1}^{n-1} b_m^{(i)\to j} \Delta t^m}{1 + \sum_{m=1}^{n-1} b_m \Delta t^m}
    \label{eq:mij_dt}
\end{equation}
where $\delta_{ij}$ is the Kronecker delta, $b_m$ is the sum over weights of branchings in $\Gamma_\Lambda$ with $m$ arcs, and $b_m^{(i)\to j}$ is the sum over weights of branchings in $\Gamma_\Lambda$ that are rooted at vertex $i$ and have a path to vertex $j$.

\begin{remark}
   Every arborescence in $\Gamma_{\cal M}$ in the above discussion must have an in arc to vertex $j$, since it is not the root vertex.  One may trace back from vertex $j$ along predecessor in arcs in the arborescence to one and only one rooted vertex $i$.  Thus, each arborescence in $\Gamma_{\cal M}$ has exactly one path from vertex $i$ to vertex $j$.  Thus, by summing over a column of ${\cal M}^{-1}$, one sums over all possible rooted vertices with a path to vertex $j$ and thus all possible arborescences and thereby obtains a column sum 1 for each column.  Each element of ${\cal M}^{-1}$ is non-negative (since all rates $\lambda_{ij}$ and $\Delta t$ are non-negative) and is less than or equal to 1 since it is a fractional weight of arborescences.  ${\cal M}^{-1}$ is therefore a column-stochastic Markov matrix.
    \label{remark:markov}
\end{remark}
The Markov matrix ${\cal M}^{-1}$ in Eq. (\ref{eq:finite_sol}) or (\ref{eq:mij_dt}) gives the transition of probabilities of states at time $t$ to those at time $t + \Delta t$, and the fractional weight of arborescences in $\Gamma_{\cal M}$ with a path from vertex $i$ to vertex $j$ gives the ``flow'' of probability from state $j$ to state $i$ from $t$ to $t + \Delta t$.

If the rates $\lambda_{ij}$ are constant in time, the system will evolve over long time to a fixed point (an equilibrium).

\begin{proposition}
    Let $\Gamma_\Lambda$ be the matrix digraph in Remark \ref{remark:Lambda}.  Let $\Gamma_\Lambda^{(k)}$ be graph derived from $\Gamma_\Lambda$ by first removing all in arcs to vertex $k$ and then adding an arc $(0, k)$ with weight 1. The equilibrium probability of state $i$ is given by
    \begin{equation}
    X_i^{(eq)} = \frac{\sum_{B \in {\cal A}(\Gamma_\Lambda^{(i)})}W(B)}{\sum_k \sum_{B \in {\cal A}(\Gamma_\Lambda^{(k)})} W(B)}
    \end{equation}
    where ${\cal A}(\Gamma_\Lambda^{(k)})$ denotes the set of all arborescences in digraph $\Gamma_\Lambda^{(k)}$ and $W(B)$ is the weight of arborescence $B$.
    \label{prop:equil}
\end{proposition}
\begin{proof}
The equilibrium occurs when
\begin{equation}
    \Lambda X = 0
    \label{eq:equil}
\end{equation}
This equation is difficult to solve directly since $\Lambda$ is singular; nevertheless, there is a constraint one needs to apply, namely, that the sum of the probabilities should be unity.  Consider the $n \times n$ matrix $U$ defined as
\begin{equation}
U_{ij} = 
\begin{cases}
1, & i = 1, \forall j\\
0, & 1 < i <= n, \forall j
\end{cases}
\label{eq:U_elements}
\end{equation}
One now finds
\begin{equation}
    \left(\Lambda + U\right)X = V
    \label{eq:U_equil}
\end{equation}
where $V$ is a column vector with $V_1 = 1$ and $V_k = 0$ for $1 < k \leq n$.  The equilibrium probabilities are then given by $\left(\Lambda + U\right)^{-1} V$.

The determinant of $\Lambda + U$ is 
the sum of determinants of all matrices formed from permutations of columns of $\Lambda$ and $U$.  $\Lambda$ itself is singular, so it does not contribute to the sum.  Furthermore, if we include more than one column of $U$ in the permutation, the determinant of the resulting matrix will be zero since there will be two columns that are multiples of each other.  The only permutations contributing to $det(\Lambda + U)$, then, will have exactly one column from $U$ and the remaining columns from $\Lambda$.  This means
\begin{equation}
    det(\Lambda + U) = \sum_{k=1}^n det(\Lambda^{\{1\},\{k\}})
    \label{eq:detLamU}
\end{equation}
that is, the sum over determinants of reduced matrices $\Lambda^{\{1\},\{k\}}$.  By Theorem \ref{theorem:reduced}, $det(\Lambda^{\{1\}, \{k\}}$ is the sum over all arborescences of the matrix digraph $\Gamma_\Lambda$ that must also have an arc $(0, k)$ with weight 1 and a path from vertex $k$ to vertex $1$.  The matrix digraph $\Gamma_\Lambda$ that must also include the arc $(0, k)$ is $\Gamma_\Lambda^{(k)}$.  Since there are no other root arcs in $\Gamma_\Lambda^{(k)}$, the contributing arborescences to $det(\Lambda^{\{1\},\{k\}})$ must be single-root arborescences, which necessarily have a path to vertex 1. As a result, we may write
\begin{equation}
    det(\Lambda^{\{1\},\{k\}}) = \sum_{B\in {\cal A}(\Gamma_\Lambda^{(k)})} W(B)
    \label{eq:detLamUk}
\end{equation}
where ${\cal A}(\Gamma_\Lambda^{(k)})$ is the set of all arborescences in $\Gamma_\Lambda^{(k)}$ and $W(B)$ is the weight of arborescence $B$.

Since the only non-zero element of vector $V$ in Eq. (\ref{eq:U_equil}) is the first, which has value 1, 
\begin{equation}
    X_i^{(eq)} = \left(\Lambda + U\right)_{i,1}^{-1}
\end{equation}By Remark \ref{remark:inverse},
\begin{equation}
    \left(\Lambda + U\right)_{k,1}^{-1} = \frac{det(\Lambda^{\{1\},\{k\}})}{det(\Lambda + U)}
    \label{eq:X_sol_eq}
\end{equation}
Substitution of Eqs. (\ref{eq:detLamU}) and (\ref{eq:detLamUk}) into Eq. (\ref{eq:X_sol_eq}) completes the proof.
\end{proof}
\begin{remark}
The equilibrium probability of state $i$ in the system modeled by Eq. (\ref{eq:dXdt}) with constant rates is the fractional weight of arborescences in $\Gamma_\Lambda$ with vertex $i$ as the root.  The same result may be inferred from Eq. (\ref{eq:mij_dt}) by considering the limit $\Delta t \to \infty$.    As $\Delta t \to \infty$, the $\Delta t^{n-1}$ term dominates, and ${\cal M}_{ij}^{-1}$ becomes the ratio of sums of single-root arborescence weights rooted at $i$ divided by the sum of all single-root arborscence weights; that is, ${\cal M}_{ij}^{-1} \to X_i^{(eq)}$.
\end{remark}

The matrix $U$ in the proof of Proposition \ref{prop:equil} implements the constraint on the sum of probabilities.  Other constraints are possible, which could be accommodated by additional constraint matrices.   For example, there may be one or more rates that are slow compared to all other transition rates.  In this case, a subset or ``cluster'' of states in the system might be equilibrated with each other but not with the remaining states.  The sum of probabilities in the states in the cluster would be $0 \leq \tilde{X} \leq 1$.  In element formation theory, this has been termed a ``quasi-equilibrium'' (e.g., \cite{1998ApJ...498..808M}).  It represents another constraint on the probabilities, and it can be accommodated, for example, by adding another matrix $W$ to $\Lambda + U$ such that $W$ has all zeros except for 1's in row 2 and in a subset of columns corresponding to the subset of states in the equilibrium cluster.  The vector $V$ then would have a 1 in row 1 and $\tilde{X}$ in row two with zeros elsewhere.  The determinant would be computed from the sum of matrices made up of permutations of columns from $\Lambda$, $U$, and $W$.  These would be a number of reduced matrices of $\Lambda$, and one would apply Theorem \ref{theorem:reduced} to compute the determinants and the inverse matrix elements.

\section{Calculation of Determinants}

By Theorem \ref{theorem:detDA}, the determinant of a matrix $A$ is the sum over weights of all arborescences of the matrix digraph corresponding to $A$.  Related graphical strategies are to compute the determinant from a sum over weights of functional digraphs \cite{MOON1994163} or over the permutation graphs on the $n$ vertices in the digraph related to the matrix \cite{zeilberger1985combinatorial}.  An advantage of the use of arborescences is that well established algorithms exist to compute $k$-th best arborescences \cite{Camerini1980TheKB, sorensen2005}.  These can be used to compute all arborescences or a partial sum over the arborescences that contribute most to the determinant.

Direct calculation of all or most of the arborescences contributing to the determinant of a matrix is limited, however, by the number of arborescences present.  For a simple, complete digraph with $n$
vertices, the underlying graph has $n^{n-2}$ spanning trees \cite{Cayley1889}.  Each spanning tree has $n$ vertices, each of which can serve as the root of an arborescence; thus, a simple, complete digraph with $n$ vertices has $n^{n-1}$ arborescences.  The matrix digraph has $n+1$ vertices since we added a root vertex.  If the digraph is simple and complete except for vertex 0, which only has out arcs, it will have $(n+1)^{n-1}$ arborescences because vertex 0 must be the root of all arborescences.  Even a modest digraph with 9 vertices (excluding the root vertex) will have $10^8$ arborescences.

Another approach is a recursive calculation.  A basic example is provided by a tridiagonal matrix $T$ given by
\begin{equation}
T = \begin{bmatrix}
a_1 & b_1 & 0 & \dots & 0 \\
c_1 & a_2 & b_2 & \ddots & \vdots \\
0 & c_2 & a_3 & \ddots & 0 \\
\vdots & \ddots & \ddots & \ddots & b_{n-1} \\
0 & \dots & 0 & c_{n-1} & a_n
\end{bmatrix}
\label{eq:trid}
\end{equation}
The determinant of the tridiagonal matrix $T$ in Eq. (\ref{eq:trid}) may be computed recursively is via continuants \cite{muir2003treatise}.
The recurrence relation for the continuant $K_{i+1}$ is
\begin{equation}
    K_{i+1} = a_{i+1} K_i - b_i c_i K_{i-1}
    \label{eq:continuant}
\end{equation}
with initial conditions $K_0 = 1$ and $K_{-1} = 0$.  $K_i$ is the determinant of the $i \times i$ sub-matrix of $T$ in Eq. (\ref{eq:trid}) with rows and columns ranging from 1 to $i$.  The determinant of the full matrix is $K_n$.

We seek a related strategy based on Theorems \ref{theorem:detDA} and \ref{theorem:reduced}.
We connect the tridiagonal matrix in Eq. (\ref{eq:trid}) with a matrix digraph by the equations
\begin{equation}
    b_i = -v_{i, i+1},\ \ \ \  1 \leq i < n
    \label{eq:trid_b}
\end{equation}
\begin{equation}
    c_i = -v_{i+1, i},\ \ \ \  1 \leq i < n,
    \label{eq:trid_c}
\end{equation}
and
\begin{equation}
a_i = 
\begin{cases}
v_{1,1} + v_{2,1}, & i = 1\\
v_{n-1, n} + v_{n,n}, & i = n\\
v_{i-1, i} + v_{i,i} + v_{i+1, i}, & 1 < i < n
\end{cases}
\label{eq:trid_ai}
\end{equation}
The resulting matrix digraph is shown in Fig. \ref{fig:trid}

\begin{center}
\begin{figure}[htb]
\centering
    \includegraphics[width=\textwidth]{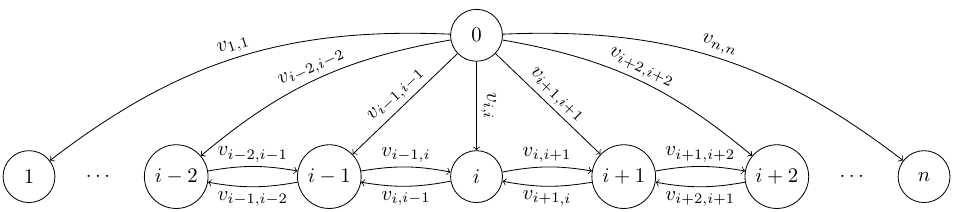}
\caption{Matrix digraph for a tridiagonal matrix of dimension $n$.}
\label{fig:trid}
\end{figure}
\end{center}

\begin{proposition}
    The determinant of the tridiagonal matrix in Eq. (\ref{eq:trid}), with the matrix digraph shown in Fig. \ref{fig:trid} and with arc weights given by the relations in Eqs. (\ref{eq:trid_b})-(\ref{eq:trid_ai}) can be computed recursively by the equations
    \begin{equation}
        D_{i+1} = \left(v_{i+1,i+1} + v_{i,i+1}\right) D_i + v_{i+1, i+1}
        v_{i+1, i} D_{(i)}
        \label{eq:Dip}
    \end{equation}
    and
    \begin{equation}
        D_{(i+1)} = D_i + v_{i+1,i} D_{(i)}
        \label{eq:Drootip}
    \end{equation}
    with $D_1 = v_{1,1}$ and $D_{(1)} = 1$.  The determinant is given by $D_n$.
\end{proposition}
\begin{proof}
    To compute the determinant from the digraph in Fig. \ref{fig:trid}, consider a subgraph $G_i$ including the vertices from 1 to $i < n$ and arcs between those vertices (and root vertex 0).  Now add the vertex $i+1$ to create the graph $G_{i+1}$.  The new arcs will be $(0, i+1)$, $(i, i+1)$, and $(i+1,i)$, as shown in Fig. \ref{fig:trid_i} (we recall that we take the arc $(0, i)$ to have weight $v_{ii}$).  Suppose $D_i$ is the determinant computed from the sum of arborescence weights for $G_i$.  The determinant corresponding to the graph $G_{i+1}$ is given by Eq. (\ref{eq:Dip}), where $D_{(i)}$ is the sum over arbosescence weights of graph $G_{(i)}$, which is a subgraph of $G_i$ that includes no in arcs to vertex $i$.  We may alternatively think of $G_{(i)}$ as a subgraph of $G_i$ that is ``rooted'' at vertex $i$; that is, that has an arc $(0,i)$ with weight 1.  An in arc to $i$ can replace the weight 1 $(0,i)$ arc.
    
    The first term on the right side of Eq. (\ref{eq:Dip}) arises from the fact that any arborescence in $G_i$ can be extended with an arc to vertex $i+1$.  The second term comes from arborescences that must include the root arc $(0, i+1)$ and then the arc $(i+1,i)$ back into $G_i$.  This extends branchings in $G_i$ that are rooted at vertex $i$.  In a similar fashion, one may compute the determinant corresponding to $G_{(i+1)}$ to be given by Eq. (\ref{eq:Drootip}) by considering the graph that must include the weight 1 arc $(0, i+1)$.  With the conditions that $D_1 = v_{1,1}$ and $D_{(1)} = 1$, Eqs. (\ref{eq:Dip}) and (\ref{eq:Drootip}) can be solved together recursively to find $D_n$, which is the determinant of matrix $T$.
\end{proof}

\begin{center}
\begin{figure}[htb]
\centering
    \includegraphics[width=0.4\textwidth]{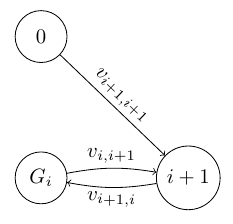}
\caption{Matrix digraph $G_{i+1}$ derived from predecessor graph $G_i$.}
\label{fig:trid_i}
\end{figure}
\end{center}

The recursive determinant calculations embodied in either Eqs. (\ref{eq:Dip}) and (\ref{eq:Drootip}) or Eq. (\ref{eq:continuant}) have cost linear in $n$ as opposed to cubic in $n$ for a general matrix.  The continuant method of Eq. (\ref{eq:continuant}) is a three-term recurrence (terms $K_{i-1}$, $K_i$, and $K_{i+1}$).  The arboresence method of Eqs. (\ref{eq:Dip}) and (\ref{eq:Drootip}) is a less efficient four-term recurrence (terms $D_i$, $D_{(i)}$, $D_{i+1}$, and $D_{(i+1)}$), although effectively it is only a three-term recurrence since we only need to keep track of the current determinant value; that is, we can update $D_{(i+1)}$ with $D_i$ and $D_{(i)}$ in Eq. (\ref{eq:Drootip}).  We can then compute $D_{i+1}$ from Eq. (\ref{eq:Dip}) but store the value in $D_i$ for the next iteration.

The principal difference between the two recursive methods comes from the treatment of the $i$-th diagonal element of the tridiagonal matrix $T$.  This may be understood graphically.  For the method based on arborescences, the subgraph $G_i$ does not include the arc $(i+1, i)$; thus, $G_i$ describes an $i \times i$ sub-matrix of $T$ that is its own independent tridiagonal matrix, and the $i$-th diagonal element does not include the weight of the arc $(i+1, i)$.  The determinant $D_i$ of the matrix corresponding to $G_i$ is the sum over arborescence weights of $G_i$.  Eq. (\ref{eq:Dip}) appropriately extends $D_i$ to $D_{i+1}$ when vertex $i+1$ and arcs $(i, i+1)$, $(i+1,i)$, and $(0,i+1)$ are added to extend graph $G_i$ to $G_{i+1}$.

For the method of continuants, the diagonal element $a_i$ includes the weight of the arc $(i+1, i)$.  In this case, the corresponding graph $G_i' = G_i \cup (i+1,i)$.  The continuant $K_i$ for $i < n$ is then the sum over arborescence weights on $G_i$ plus the sum over weights of branchings formed by replacing the in arc to vertex $i$ in any of the arborescences of $G_i$ with the arc $(i+1, i)$.  This is the determinant of the $i \times i$ sub-matrix of $T$ in Eq. (\ref{eq:trid}) containing rows and columns 1 through $i$.

When one extends the calculation to $K_{i+1}$, one adds the arcs $(i,i+1,)$, $(0, i+1)$, and $(i+2, i+1)$ to $G_i'$ to create the graph $G_{i+1}'$.  Addition of these arcs to $K_i$ creates new branchings contributing to $K_{i+1}$; however, it also creates a set of subgraphs with the cycle $(i+1,i)-(i, i+1)$.  These must be subtracted off to ensure only branchings are present.  This is the origin of the second term on the right side of Eq. (\ref{eq:continuant}).  This issue does not occur for vertex $n$ since there is no arc $(n+1, n)$.  Thus, while $K_i \ne D_i$ for $i < n$, $D_n = K_n$, and the two methods give the same result for the determinant of $T$.

\begin{remark}
    The graphical picture of subgraphs and subtracted cycles in the calculation of continuants may be extended to general matrices.  Consider an $n\times n$ matrix $A$ with elements given by Eq. (\ref{eq:aij}).  The standard formula for the determinant of $A$ is
    \begin{equation}
        det(A) = \sum_{i_1, i_2, ..., i_n} \epsilon_{i_1i_2...i_n} a_{1i_1} a_{2i_2}...a_{ni_n}
        \label{eq:leibnitz}
    \end{equation}
    with the sum taken over all $n$-tuples $(i_1, i_2, ..., i_n)$ and $\epsilon_{i_1i_2...i_n}$ the usual Levi-Civita symbol.  In the matrix digraph $G$ corresponding to $A$, the element $a_{ii}$ is the sum of the weights of all arcs into vertex $i$.  If we define the weight of a subgraph of $G$ to be the product of the weights of all arcs in the subgraph, the term $a_{11} a_{22} ... a_{nn}$ in Eq. (\ref{eq:leibnitz}) corresponds to the sum over weights of all subgraphs of $G$ with indegree exactly one for each vertex except the root vertex $0$ (the source of the arc with weight $v_{ii}$).  This set of subgraphs will include all subgraphs with combinations of simple cycles plus all possible arborescences in $G$.  From Theorem \ref{theorem:detDA},  the remaining terms in Eq. (\ref{eq:leibnitz}) must subtract off the subgraphs with cycles leaving only the subgraphs that are arborescences.
\end{remark}

The recursive calculation in Eqs. (\ref{eq:Dip}) and (\ref{eq:Drootip}) works by creating new arborescences on top of existing arborescences without the need to subtract off cycles.  While this method is efficient for tridiagonal matrices, it becomes less so for more complicated matrices.  Consider, for example, a pentadiagonal matrix.  In this case, one must keep track of $D_i$, the sum over arborescences of graph $G_i$, $D_{(i)}$, the sum over weights of arborescence of the subgraph of $G_i$ rooted at $i$ (with arc $(0,i)$ having weight 1), $D_{(i-1)}$, the sum over weights of arborescences of the subgraph of $G_i$ rooted at $i-1$, and $D_{(i-1, i)}$, the sum over weights of arborescences of the subgraph of $G_i$ rooted at $i-1$ and $i$.  In addition, one needs to keep track of $D_{(i)\not\to i-1}$ and $D_{(i-1)\not\to i}$.  The former of these two quantities is the sum over weights of arborescences rooted at $i$ but not having a path to $i-1$ while the latter is the sum over weights of arborescences rooted at $i-1$ but not having a path to $i$.  These are needed because, in adding vertex $i+1$, one can have, for example, an arborescence that has an arc $(i+1, i)$ and then $(i-1, i+1)$.  A new arborescence  can only occur by adding these arcs to $D_{(i)\not\to i-1}$, otherwise, there would be a cycle.  With these definitions, and our notion that $v_{i,j}$ is the weight of the arc $(i,j)$ in $G_i$ (except $v_{i,i}$ is the weight of arc $(0,i)$), we may find the following recurrence relations:
\begin{equation}
        \begin{split}
        D_{i+1} = \left(v_{i-1,i+1} + v_{i, i+1} + v_{i+1, i+1}\right) D_i + v_{i+1,i+1} v_{i+1,i} D_{(i)}\\
        + v_{i+1,i+1} v_{i+1, i-1} D_{(i-1)} + v_{i+1, i+1} v_{i+i, i} v_{i+1, i-1} D_{(i-1, i)}\\ + v_{i+1, i} v_{i-1,i+1} D_{(i)\not\to i-1} + v_{i+1, i-1} v_{i, i+1} D_{(i-1)\not\to i} 
        \end{split}
        \label{eq:D_p_i}
\end{equation}
\begin{equation}
    D_{(i+1)} = D_i + v_{i+1,i} D_{(i)} + v_{i+1,i-1} D_{(i-1)}\\ + v_{i+1,i-1} v_{i+1,i} D_{(i-1,i)}
    \label{eq:D_p_p_p}
\end{equation}
\begin{equation}
    \begin{split}
    D_{(i)} = \left(v_{i-1,i+1} + v_{i,i+1} + v_{i+1,i+1}\right)D_{(i)}\\
    + v_{i+1,i-1} \left(v_{i,i+1} + v_{i+1,i+1}\right) D_{(i-1, i)} 
    \end{split}
    \label{eq:D_p_p_i}
\end{equation}
\begin{equation}
    D_{(i, i+1)} = D_{(i)} + v_{i+1, i-1} D_{(i-1,i)}
\end{equation}
\begin{equation}
    D_{(i+1)\not\to i} = D_i + v_{i+1,i-1} D_{(i-1)\not\to i}
\end{equation}
\begin{equation}
    D_{(i)\not\to i+1} = v_{i+1,i+1} D_{(i)} + v_{i-1,i+1} D_{(i)\not\to i-1} + v_{i+1,i+1} v_{i+1, i-1} D_{(i-1,i)}
\end{equation}
The initial values are $D_1 = v_{1,1}$, $D_{(1)} = 1$, and $D_{(0)} = D_{(0,1)} = D_{(1)\not\to 0} = D_{(0)\not\to 1} = 0$.  This method is in fact a 12-term recurrence since the $D_{(i)}$ on the left side of Eq. (\ref{eq:D_p_p_i}) is the updated version of $D_{(i-1)}$ on the right side of Eqs. (\ref{eq:D_p_i}) and (\ref{eq:D_p_p_p}), although we would only need to keep track of 11 terms since we can update the current determinant with Eq. (\ref{eq:D_p_i}) at the end of each iteration using the un-updated values for the other terms.

A six-term recurrence for computing determinants of pentadiagonal matrices exists \cite{sweet1969recursive}, and there are other highly efficient algorithms \cite{evans1975recursive, SOGABE2008835}.  Again, from a graphical picture, these methods compute subgraphs with all vertices having indegree no greater than one and then subtracting off cycles.  Our method requires more terms because it must explicitly keep track of arborescence weights on rooted subgraphs of the matrix digraph.

For even larger diagonal matrices, the number of recurrence relations for our graphical approach becomes even larger.  One must keep track of the various sums of arborescence weights of rooted subgraphs at each iteration, including those of subgraphs that do not have a path from the rooted vertex to another particular vertex.  While this is a tedious procedure, the general approach is straightforward, and we plan to develop it more explicitly in a future work.

Given the number and complexity of the recurrence relations for a matrix with a large number of diagonals, a recursive approach to computing determinants that builds arborescences only on top of existing arborescences is likely not useful except in special sparse matrix cases.  Nevertheless, these considerations give an interesting picture of matrix determinants and how they can be built up from sums of arborescence weights over sub-graphs of the matrix digraph. 

The recursive approach can also provide valuable insights into physical processes modeled by matrices.  For example, in the time evolution of a $n$-state system discussed in \S \ref{sec:application}, the flow of probability of state $j$ at time $t$ to that of state $i$ at time $t + \Delta t$ was given by the fractional weight of arborescences with a path from vertex $i$ to $j$.
We have been able to use matrix-tree theorem based recursive strategies on tridiagonal matrices similar to those in Eqs. (\ref{eq:Dip}) and (\ref{eq:Drootip}) to compute inverse matrix elements relevant for heavy-element nucleosynthesis \cite{2025ApJ...982..139L}.  This approach allowed us to define effective rates that account for cycling among species over time step $\Delta t$.  This gave a clear picture of the flow from one species to another over time in the network.

We have developed a GitHub repository of Python codes to illustrate some of the content of this section \cite{Ghosh_Matrix_Digraph_2023}.  The repository includes codes to compute matrix determinants via the $k$-th best arborescence algorithm and the pentadiagonal matrix recursive strategy described above.  There is also a code to generate random matrices to use with the determinant codes.

\section{Conclusion}
We have provided proofs and explicit examples for the weighted digraph versions of the matrix-tree and matrix-forest theorems.  While these theorems are already well known (for example, Theorem 3.1 and, especially, Corollary 4.1 of Moon \cite{MOON1994163}), our versions of the theorems with an extra root vertex are particularly useful for calculations of the determinant of a matrix via $k$-th best arborescence algorithms.  They can also be helpful for developing other strategies for computing determinants, such as the recursive method we briefly outlined, and for interpreting physical processes modeled by matrices.

\section{Acknowledgments}

This work was supported by NASA Emerging Worlds grant 80NSSC20K0338.

\bibliographystyle{plain}
\bibliography{clemson}

\end{document}